\documentclass[11pt,a4paper]{article}
\usepackage{gfsartemisia-euler}
\usepackage{times}
\usepackage[T1]{fontenc}
\usepackage[utf8x]{inputenc}
\usepackage{amsmath}
\usepackage{amsfonts}
\usepackage{amssymb}
\usepackage{amsthm}
\usepackage{graphicx}
\usepackage{color}
\usepackage{enumitem}
\usepackage[width=16.00cm, height=24.00cm]{geometry}

\usepackage[colorlinks=true,linkcolor=colorref,citecolor=colorcita,urlcolor=colorweb]{hyperref}
\definecolor{colorcita}{RGB}{21,86,130}
\definecolor{colorref}{RGB}{5,10,177}
\definecolor{colorweb}{RGB}{177,6,38}

\newtheorem{theorem}{Theorem}[section]
\newtheorem{lemma}[theorem]{Lemma}
\newtheorem{proposition}[theorem]{Proposition}
\newtheorem{corollary}[theorem]{Corollary}

\theoremstyle{definition}
 
\newtheorem{remark}[theorem]{Remark}

\newcommand{\punkt}{\,\begin{picture}(-1,1)(-1,-3)\circle*{2.5}\end{picture}\;\; }

\newcommand{\T}{\mathbb{T}}
\newcommand{\N}{\mathbb{N}}
\newcommand{\C}{\mathbb{C}}
\newcommand{\E}{\mathbb{E}}

\DeclareMathOperator{\mon}{mon}

\begin{document}
\title{Hausdorff-Young type inequalities for vector-valued Dirichlet series}

\author{Daniel Carando \thanks{Departamento de Matem\'atica, Facultad de Cs. Exactas y Naturales,
		Universidad de Buenos Aires and IMAS-UBA-CONICET, Int.~G\"uiraldes s/n, 1428, Buenos Aires, Argentina (dcarando@dm.uba.ar). Supported by CONICET-PIP 11220130100329CO and ANPCyT PICT 2015-2299.}\and
Felipe Marceca\thanks{Departamento de Matem\'atica, Facultad de Cs. Exactas y Naturales,
	Universidad de Buenos Aires and IMAS-UBA-CONICET, Int.~G\"uiraldes s/n, 1428, Buenos Aires, Argentina (fmarceca@dm.uba.ar). Supported by a CONICET doctoral fellowship,  CONICET-PIP 11220130100329CO and ANPCyT PICT 2015-2299.} \and
Pablo Sevilla-Peris\thanks{Instituto Universitario de Matem\'atica Pura y Aplicada,
Universitat Polit\`{e}cnica de Val\`encia, cmno Vera s/n, 46022,
Val\`encia, Spain (psevilla@mat.upv.es) Supported by MINECO and FEDER projects MTM2014-57838-C2-2-P and MTM2017-83262-C2-1-P}}

\date{}

\maketitle

\begin{abstract}

We study Hausdorff-Young type inequalities for vector-valued Dirichlet series which allow to compare the norm of a Dirichlet series in the Hardy space $\mathcal{H}_{p} (X)$ with the $q$-norm of its coefficients.
In order to obtain inequalities completely analogous to the scalar case, a Banach space must satisfy the restrictive notion of Fourier type/cotype. We show that  variants of these inequalities hold for the much broader range of spaces enjoying type/cotype.
We also consider Hausdorff-Young type inequalities for functions defined on the infinite torus $\mathbb{T}^{\infty}$ or the boolean cube $\{-1,1\}^{\infty}$.
\end{abstract}

\thispagestyle{empty}

\section{Introduction}

The Hilbert space of Dirichlet series $\mathcal{H}_{2}$ was first defined in  \cite{HeLiSe97} as those  $\sum a_{n} n^{-s}$ for which $(a_{n})_{n} \in \ell_{2}$. This was later extended by Bayart, who in
\cite{Ba02} defined a whole scale of Hardy spaces of Dirichlet series $\mathcal{H}_{p}$ for $1 \leq p \leq \infty$. Unlike the Hilbert space case, there is no general principle that allows to decide whether or not a
Dirichlet series belongs to a given Hardy space just by looking at the size of the coefficients, but the classical Haussdorff-Young inequalities are a useful tool in this purpose. For each $1 \leq p \leq \infty$ the
spaces $\mathcal{H}_{p}$ and $H_{p}(\mathbb{T}^{\infty})$ (precise definitions are given below) are isometrically isomorphic. A rather straightforward computation (using, for example, standard interpolation arguments) shows that Hausdorff-Young inequalities also hold for these
spaces and this immediately gives (here $r'$ denotes the conjugate of $1 \leq r \leq \infty$ so that $\frac{1}{r} + \frac{1}{r'}=1$)
\begin{equation} \label{tragica}
 \Big\Vert \sum a_{n} n^{-s} \Big\Vert_{\mathcal{H}_{p'}}
\leq C \Big( \sum_{n =1}^{\infty} \vert a_{n} \vert^{p} \Big)^{\frac{1}{p}}
\end{equation}
for every $1 \leq p \leq 2$ and
\begin{equation} \label{grande}
\Big( \sum_{n=1}^{\infty} \vert a_{n} \vert^{q} \Big)^{\frac{1}{q}}
\leq C \Big\Vert \sum a_{n} n^{-s} \Big\Vert_{\mathcal{H}_{q'}}
\end{equation}
for all $2 \leq q \leq \infty$.

Hardy spaces $\mathcal{H}_{p}(X)$ of vector-valued Dirichlet series (that is, the coefficients $a_{n}$ belong to some Banach space $X$) have been defined and studied in \cite{CaDeSe14,DePe18}. Here the problem becomes more complicated. Once again, each one of these spaces is isometrically isomorphic to the corresponding $H_{p}(\mathbb{T}^{\infty}, X)$, but in this case the Haussdorff-Young inequalities do not hold for an arbitrary Banach space. Fourier type and cotype are the notions to get vector-valued Haussdorff-Young inequalities, and for spaces enjoying those properties (again, see below for the definition) we easily get in Propositions~\ref{proposition3-cotipo} and~\ref{proposition3}  inequalities that are analogous to  \eqref{tragica} and \eqref{grande}.
However, these properties are very restrictive in the sense that a Banach space has Fourier type or cotype with exponents $p$ or $q$ which are generally worse than those for the usual (Rademacher) type and cotype. Furthermore, the actual values of $p$ and $q$ are often unknown.

In Theorems~\ref{ponzio} and~\ref{maidana} we show that Banach spaces with  cotype $q$ (respectively type $p$) satisfy variants of Hausdorff-Young  inequalities which relate the  norm of a Dirichlet series with a weighted $\ell_q$ norm of the coefficients (respectively, a weighted $\ell_p$ norm). Analogous inequalities are obtained for functions on $\mathbb{T}^{\infty}$ and Walsh functions (i.e., functions on the infinite Boolean cube).

The main tool for these results is a polynomial reformulation  of type and cotype (Theorems~\ref{cotimplicapolcot} and \ref{typimplicapoltyp}).
More precisely, we prove that Rademacher cotype is equivalent to the notion of {hypercontractive homogeneous cotype} defined in \cite{CaDeSe16} (see Section~\ref{sec-tyc_pol} for the definition). An equivalence with an inequality concerning Walsh polynomials is also established. Analogously,  Theorem~\ref{typimplicapoltyp} shows the corresponding results for type and its hypercontractive homogeneous version. In~\cite{DeMaSc_12} variants of vector-valued Bonhenblust-Hille inequality with operators are shown to hold for Banach lattices nontrivial cotype. In~\cite{CaDeSe16}, results regarding monomial convergence sets and multipliers for Hardy spaces were presented for Banach spaces with nontrivial cotype and local unconditional structure or with Fourier cotype. As mentioned in Remark~\ref{loquevaleahora}, thanks to Theorem ~\ref{cotimplicapolcot} all these results readily extend to Banach spaces with nontrivial cotype.

The proof of Theorems~\ref{cotimplicapolcot} and~\ref{typimplicapoltyp} are the most technical part of the article and are developed in Section~\ref{sec1}. A crucial part is to show that it suffices to prove the desired inequality for tetrahedral polynomials, where the variables appear with at most power $1$. We feel that this methodology (to reduce an inequality for general polynomials to an inequality for tetrahedral polynomials) can be useful in different situations and is, then, interesting on its own.

\section{Definitions and first results}

We denote by $dz$ the normalized Lebesgue measure on the  infinite dimensional polytorus $\mathbb{T}^{\infty} = \prod_{k=1}^\infty \mathbb{T}$, i.e., the countable product measure  of the normalized Lebesgue measure
on $\mathbb{T}$. For any  multi index $\alpha = (\alpha_1, \dots, \alpha_n,0, \ldots ) \in \mathbb{Z}^{(\mathbb{N})}$ (all finite sequences in $\mathbb{Z}$)
the $\alpha$th Fourier coefficient $\hat{f}(\alpha)$ of
  $f  \in L_{1} (\mathbb{T}^{\infty}, X)$ is given by
\[
\hat{f}(\alpha) = \int_{\mathbb{T}^{\infty}} f(z) z^{- \alpha} dz\,.
\]
For $1 \le p < \infty$, the $X$-valued Hardy space on $\mathbb{T}^\infty$ is the subspace of $L_{p} (\mathbb{T}^{\infty}, X)$ defined as
\[
H_{p}(\mathbb{T}^{\infty},X) = \Big\{ f \in  L_{p} (\mathbb{T}^{\infty},X) \,:\,\, \hat{f}(\alpha) = 0 \,  , \, \,\, \,
\forall \alpha \in \mathbb{Z}^{(\mathbb{N})} \setminus \mathbb{N}_{0}^{(\mathbb{N})} \Big\}
\]
(where $ \mathbb{N}_{0}^{(\mathbb{N})}$ stands for the set of $\alpha$s in $\mathbb{Z}^{(\mathbb{N})}$ with $\alpha_{i} \geq 0$ for every $i$).
Observe that each $f \in H_{p}(\mathbb{T}^{\infty},X)$  is uniquely determined by its Fourier coefficients. With this in mind we consider the $X$-valued Bohr transform $\mathfrak{B}_X$ that to each $f$ assigns the Dirichlet series $\sum a_{n} n^{-s}$ where $a_{n} = \hat{f}(\alpha)$ if $n = p_{1}^{\alpha_{1}} \cdots p_{k}^{\alpha_{k}}$ is the prime number decomposition of $n$.
Then the Hardy space $\mathcal{H}_p(X)$ of Dirichlet series in $X$ is defined as  the  image of $H_{p}(\mathbb{T}^{\infty},X)$ under the Bohr transform $\mathfrak{B}_X$.
This vector space of Dirichlet series together with the norm
\begin{equation} \label{armani}
\|D\|_{\mathcal{H}_p(X)}= \|\mathfrak{B}_X^{-1}(D)\|_{H_{p}(\mathbb{T}^{\infty},X)}
\end{equation}
forms a Banach space. In other words, Bohr's transform  gives the isometric identification
\[
\mathcal{H}_p(X) = H_{p}(\mathbb{T}^{\infty},X)\, \text{ for } 1 \le p < \infty.
\]
A detailed account on this identification can be found in  \cite{DeSe19_libro} or \cite{QuQu13}.

There are many equivalent definitions of Fourier type and cotype (see \cite{GCKaKoTo98}). Let us give the ones that are more akin to our framework. Given $1 \leq p \leq 2$, we say that $X$
has \emph{Fourier type} $p$ if there is a constant $C > 0$
such that for each choice of finitely many vectors $x_1, \ldots, x_N \in X$ we have
\[
\Big( \int_{\mathbb{T}} \Big\| \sum_{k=1}^N  x_k z^{k}  \Big\|^{p'} dz \Big)^{\frac{1}{p'}}
\leq C \Big(\sum_{k=1}^N \big\| x_k \big\|^p \Big)^{\frac{1}{p}}   \,.
\]
For   $2 \leq q < \infty$, $X$ has
\emph{Fourier cotype} $q$ if there is a constant $C > 0$ such that for each choice of finitely many vectors $x_1, \ldots, x_N \in X$ we have
\[
\Big(\sum_{k=1}^N \big\| x_k \big\|^q \Big)^{\frac{1}{q}}  \leq C \Big( \int_{\mathbb{T}} \Big\| \sum_{k=1}^N  x_k z^{k}  \Big\|^{q'} dz \Big)^{\frac{1}{q'}} \,.
\]

We refer to the comments after Proposition~\ref{proposition3} regarding the equivalence of these two concepts and also their connection with \eqref{eq-analytic-cotype} and \eqref{eq-analytic-type} below.
It was shown in  \cite[Proposition~2.4]{CaDeSe16} that a Banach space $X$
has Fourier cotype $q\ge 2$ if and only if there exists $C >0$ such that for every finite family $(x_{\alpha})_{\alpha  \in \mathbb{N}_{0}^{(\mathbb{N})}}$ we have
\begin{equation}\label{eq-analytic-cotype}
 \Big( \sum_{ \alpha  } \Vert x_{\alpha} \Vert^{q} \Big)^{\frac{1}{q}} \leq C \Big( \int_{\mathbb{T}^{n}} \Big\Vert  \sum_{ \alpha }  x_{\alpha} z^{\alpha} \Big\Vert^{q'} dz \Big)^{\frac{1}{q'}} \,.
\end{equation}
The proof of  \cite[Proposition~2.4]{CaDeSe16}) also works to show that  $X$
has Fourier  type $1 \leq p \leq 2$ if and only if  there exists $C >0$ such that for every finite family $(x_{\alpha})_{\alpha  \in \mathbb{N}_{0}^{(\mathbb{N})}}$ in $X$ we have
\begin{equation}\label{eq-analytic-type}
\Big( \int_{\mathbb{T}^{n}} \Big\Vert  \sum_{ \alpha }  x_{\alpha} z^{\alpha} \Big\Vert^{p'} dz \Big)^{\frac{1}{p'}}
\leq C \Big( \sum_{ \alpha  } \Vert x_{\alpha} \Vert^{p} \Big)^{\frac{1}{p}}  \,.
\end{equation}

A straightforward argument using the Bohr transform (see \eqref{armani}) allows to reformulate \eqref{eq-analytic-cotype} and \eqref{eq-analytic-type} in terms of Dirichlet series as
 \begin{equation} \label{ah}
\left(\sum_{n=1}^{N} \|a_n\|_X^{q}\right)^{1/{q}} \leq C %\|D\|_{\mathcal{H}_{q'}(X)}
\Big\Vert \sum_{n=1}^{N} a_{n} n^{-s} \Big\Vert_{\mathcal{H}_{q'}(X)}
\end{equation}
and
\begin{equation} \label{eh}
\Big\Vert \sum_{n=1}^{N} a_{n} n^{-s} \Big\Vert_{\mathcal{H}_{p'}(X)}
\leq C \left(\sum_{n=1}^{N} \|a_n\|_X^p\right)^{\frac{1}{p}},
\end{equation}
respectively, for every $X$-valued Dirichlet polynomial $\sum_{n=1}^{N} a_{n} n^{-s}$. 
Note that \eqref{eh} and the density of the finite sequences in $\ell_{p}(X)$ (the space of $p$-summing sequences in $X$) show that the operator
$\ell_{p}(X) \to \mathcal{H}_{p'}(X)$ given by $(a_{n}) \rightsquigarrow \sum a_{n} n^{-s}$ is continuous. Analogously, by \eqref{ah} and the density of the Dirichlet polynomials in
$\mathcal{H}_{q'}(X)$ (see \cite[24.2.1e]{DeSe19_libro}), the operator $\mathcal{H}_{q'}(X) \to \ell_{q}(X)$ given by $\sum a_{n} n^{-s} \rightsquigarrow (a_{n})$ is also continuous. This gives the equivalence between the first and third
statements in each of the following two results. The equivalence between the second and third statements is a straightforward consequence of the definition of the Hardy spaces of Dirichlet series.

\begin{proposition}\label{proposition3-cotipo}
Let $X$ be a Banach space. For  $2\leq q < \infty$ and $C \geq 1$, the following statements are equivalent:
 \begin{enumerate}[label=(\alph*)]
  \item \label{proposition3-cotipo a} $X$ has Fourier cotype $q$ with constant $C$;
  \item \label{proposition3-cotipo b} every Dirichlet series $D=\sum a_{n} n^{-s} \in \mathcal{H}_{q'}(X)$ satisfies
 \[
\Big(\sum_{n=1}^\infty \|a_n\|_X^{q}\Big)^{1/{q}} \leq C \|D\|_{\mathcal{H}_{q'}(X)};
\]

\item \label{proposition3-cotipo c} every $f \in H_{q'} (\mathbb{T}^{\infty},X)$ satisfies
\[
\Big(\sum_{\alpha \in  \mathbb{N}_{0}^{(\mathbb{N})}} \|\hat{f}(\alpha)\|_X^{q}\Big)^{1/{q}} \leq C \|f\|_{H_{q'} (\mathbb{T}^{\infty},X)}.
\]
\end{enumerate}
\end{proposition}

\begin{proposition}\label{proposition3}
 Let $X$ be a Banach space. For  $1\leq p \leq 2$ and $C \geq 1$, the following statements are equivalent:
 \begin{enumerate}[label=(\alph*)]
  \item $X$ has Fourier type $p$  with constant $C$;
  \item for every $(a_{n})_{n} \in \ell_{p} (X)$ the Dirichlet series $D=\sum a_{n} n^{-s}$ converges in $\mathcal{H}_{p'}(X)$ and
 \[
\|D\|_{\mathcal{H}_{p'}(X)} \leq C \Big(\sum_{n=1}^\infty \|a_n\|_X^p\Big)^{\frac{1}{p}};
\]

\item \label{proposition 3 c} for every $(x_{\alpha})_{\alpha \in \mathbb{N}_{0}^{(\mathbb{N})}} \in \ell_{p} (X)$ there is a function $f \in H_{p'}(\mathbb{T}^{\infty}, X)$ and so that $\hat{f}(\alpha) = x_{\alpha}$ for every $\alpha$ and
\[
\|f\|_{H_{p'}(\mathbb{T}^{\infty}, X)} \leq C \Big(\sum_{\alpha \in  \mathbb{N}_{0}^{(\mathbb{N})}} \|\hat{f}(\alpha) \|_X^p\Big)^{\frac{1}{p}}.
\]
 \end{enumerate}
\end{proposition}

As a matter of fact, Fourier type and cotype can be seen as particular cases in the more general theory of Fourier type with respect to groups (see \cite{GCKaKoTo98}, whose
notation we follow now, for an excellent survey on this and related subjects). Within this setting Fourier type $p$   (as we have defined it) is Fourier type $p$  with respect to  $\mathbb{Z}$, and
our Fourier cotype $q$ is Fourier type $q'$ with respect to $\mathbb{T}$. Then \cite[Theorem~6.6]{GCKaKoTo98} implies that $X$ has Fourier type $p$ if and only if it has Fourier cotype $p'$, and hence both concepts are equivalent. However, we have preferred to deal with them separately because we later work with other notions of type and cotype (which are not equivalent to each other) and in this way the relationship between these and the new ones becomes more apparent.

On the other hand, this abstract point of view allows a proof of Propositions~\ref{proposition3-cotipo} and~\ref{proposition3} based on known results on Fourier type on groups. We only sketch here the arguments.
Regarding Proposition~\ref{proposition3-cotipo}, simply note that the statement~\ref{proposition3-cotipo c} is Fourier type $q'$ with respect to  $\mathbb{T}^{\infty}$. Then the equivalence between~\ref{proposition3-cotipo a} and~\ref{proposition3-cotipo c} follows from  \cite[Theorem~6.14]{GCKaKoTo98}.

The argument for Proposition~\ref{proposition3} is slightly longer. First of all $X$ has Fourier type $p$ if and only if $X^{*}$ has Fourier type $p$ with respect to $\mathbb{T}$ \cite[Theorem~6.3]{GCKaKoTo98}, and this happens if and only if $X^{*}$  has Fourier type $p$ with respect to $\mathbb{T}^{\infty}$ by \cite[Theorem~6.14]{GCKaKoTo98}. Again by \cite[Theorem~6.3]{GCKaKoTo98}, this is equivalent to $X$ having type $p$ with respect to the dual group of $\mathbb{T}^{\infty}$, which is $\mathbb{Z}^{(\mathbb{N})}$, and this is Proposition~\ref{proposition3}--\ref{proposition 3 c}.

Propositions \ref{proposition3-cotipo} and \ref{proposition3} provide Hausdorff-Young inequalities for vector valued Dirichlet series which are analogous to the original inequalities. However as mentioned in the introduction, Fourier type (or cotype) is a very restrictive property on the geometry of a Banach space. We work with the much weaker notions of type and cotype.

A Banach space $X$ is said to have \textit{cotype} $2 \leq q < \infty$ if there is a constant $C\geq 1$ such that for every $N \in \mathbb{N}$ and every $x_{1}, \ldots , x_{N} \in X$ we have
\begin{equation} \label{defcotipo}
\Big( \sum_{n=1}^{N} \Vert x_{n} \Vert^{q} \Big)^{\frac{1}{q}}
\leq C \bigg( \int_{\mathbb{T}^{N}} \Big\Vert \sum_{n=1}^{N} x_{n} z_{n} \Big\Vert^{q} dz \bigg)^{\frac{1}{q}} \,,
\end{equation}
and \textit{type} $1 \leq p \leq 2$ if there is a constant $C\geq 1$ such that for every $N \in \mathbb{N}$ and every $x_{1}, \ldots , x_{N} \in X$ we have
\begin{equation} \label{deftipo}
\bigg( \int_{\mathbb{T}^{N}} \Big\Vert \sum_{n=1}^{N} x_{n} z_{n} \Big\Vert^{p} dz \bigg)^{\frac{1}{p}} \leq C \Big( \sum_{n=1}^{N} \Vert x_{n} \Vert^{p} \Big)^{\frac{1}{p}} \,.
\end{equation}
We denote the best constants in these inequalities by $C_{q}(X)$ and $T_{p}(X)$ respectively.
Let us note that (see e.g. \cite[Theorem~6.8]{DeSe19_libro}) the $\Vert \cdot \Vert_{L_{r}}$-norms appearing at
\eqref{defcotipo} and \eqref{deftipo}
can be replaced by any other $\Vert \cdot \Vert_{L_{s}}$-norm at the only expense of modifying the constant.

Usually, type and cotype are defined in terms of Rademacher functions. It is well known that the definitions given above are equivalent to their Rademacher versions. Actually, this equivalence can be seen as a particular case of Lemma~\ref{lemma1}, since linear combinations of Rademacher functions are just 1-homogeneous Walsh polynomials.

 For spaces with finite cotype, translating \cite[Theorem~1.1]{CaDeSe14} to our setting provides a lower estimate of the norm of a Dirichlet series in terms of its coefficients. More precisely, for a Banach space $X$ with cotype $q$, $\sigma>1/q'$ and $1\leq p \leq \infty$ there is a constant $C\geq1$ such that
\begin{align}\label{fractalosa}
\sum_{n=1}^\infty \frac{\|a_n\|_X}{n^\sigma} \leq C \|D\|_{\mathcal{H}_{p}(X)},
\end{align}
for every $D\in \mathcal{H}_{p}(X)$.
In Corollary \ref{existente}, we prove that for every $\delta>0$ there is a constant $C\geq1$ such that
\begin{align*}
%\label{fractalosa2}
\Big(\sum_{n=1}^\infty \frac{\|a_n\|^q_X}{n^\delta}\Big)^{1/q} \leq C \|D\|_{\mathcal{H}_{p}(X)},
\end{align*}
for every $D\in \mathcal{H}_{p}(X)$.
Notice that this inequality is stronger than \eqref{fractalosa} since taking $\delta=\sigma-1/q'$ and applying H\"older's inequality to the left-hand side of \eqref{fractalosa} we get
\[\sum_{n=1}^\infty \frac{\|a_n\|_X}{n^\sigma} =
\sum_{n=1}^\infty \frac{\|a_n\|_X}{n^{\delta/q}}\frac{1}{n^{(\delta+1)/q'}}\leq \Big(\sum_{n=1}^\infty \frac{\|a_n\|^q_X}{n^\delta}\Big)^{1/q}
\Big(\sum_{n=1}^\infty \frac{1}{n^{\delta+1}}\Big)^{1/q'}
\leq C\Big(\sum_{n=1}^\infty \frac{\|a_n\|^q_X}{n^\delta}\Big)^{1/q}.\]

\medskip
Following \cite[Section~5.4]{Pi16} (see also \cite[Chapter~13]{DiJaTo95}), we consider $\{-1,1\}^{\infty}$ with the probability measure given by the infinite product of the uniform probability $(\delta_{1} + \delta_{-1})/2$. For $\varepsilon = (\varepsilon_{n})_{n} \in \{-1,1\}^{\infty}$ and $A \subset \mathbb{N}$ finite we denote
\[
\varepsilon_{A} = \prod_{n \in A} \varepsilon_{n} .
\]
A finite sum $\sum_{A}  x_{A} \varepsilon_{A}$ will be called a \textit{Walsh polynomial}.
Due to the probabilistic nature of the measure space, when dealing with $L_{p}(\{-1,1\}^{\infty},X)$, we will write $\mathbb{E}$ (expected value) rather than integrals. For $f \in L_{1}(\{-1,1\}^{\infty},X)$, the corresponding Walsh-Fourier coefficients are defined by
\[
\hat{f} (A) = \E[ f (\varepsilon) \varepsilon_{A} ] \,.
\]
With this at hand we may introduce another notion of type/cotype. A Banach space $X$ has Walsh type $p$ if there is a constant $C>0$ such that for every $n$ and every family $\{ x_{A} \colon A \subseteq [n] \} \subset X$ (here and all through the text we denote $[n]= \{1, \ldots, n \}$ for each $n\in \N$) we have
\[
\Big(\E\Big\Vert \sum_{A}  x_{A} \varepsilon_{A} \Big\Vert^{p'} \Big)^{\frac{1}{p'}}
\leq C \Big( \sum_{A} \Vert x_{A} \Vert^{p} \Big)^{\frac{1}{p}} \,,
\]
and has Walsh cotype $q$ if
here is a constant $C>0$ such that for every $n$ and every family $\{ x_{A} \colon A \subseteq [n] \} \subset X$ we have
\[
\Big( \sum_{A} \Vert x_{A} \Vert^{q} \Big)^{\frac{1}{q}}
\leq C
\Big(\E\Big\Vert \sum_{A}  x_{A} \varepsilon_{A} \Big\Vert^{q'} \Big)^{\frac{1}{q'}}  \,.
\]

Standard density arguments allow us to reformulate these concepts as inequalities analogous to Proposition~\ref{proposition3}\ref{proposition 3 c} and Proposition~\ref{proposition3-cotipo}\ref{proposition3-cotipo c}. Indeed,
 $X$ has Walsh type $p$ if and only if there is $C \geq 1$ so that
\begin{equation} \label{gaita}
	 	\|f \|_{L_{p'}(\{-1,1\}^\infty,X)} \leq C \Bigg(\sum_{\substack{A \subset \mathbb{N} \\ A \text{ finite}}} \|\hat f(A)\|^p\Bigg)^{\frac{1}{p}}.
\end{equation}
Analogously, for $X$ with Walsh cotype $q$, we have
\begin{equation} \label{tamboril}
	\Bigg(\sum_{\substack{A \subset \mathbb{N} \\ A \text{ finite}}}  \|\widehat {f}(A)\|^q\Bigg)^{\frac{1}{q}}
\leq C \|f\|_{L_{q'}(\{-1,1\}^\infty,X)}.
\end{equation}

Once again, these notions of type/cotype sit in a more general framework, namely that of type/cotype with respect to an orthonormal system (we refer again to \cite{GCKaKoTo98}). The concepts of Walsh type $p$ and Walsh cotype $p'$ coincide (see \cite[Theorem~7.14]{GCKaKoTo98}). To our best knowledge it is not known whether or not these are the same as Fourier type and cotype.

\bigskip

We end this section addressing the notion of $K$-convexity which is closely related to the concepts of type/cotype. A Banach space $X$ is said to be $K$-convex if the Rademacher projection is bounded. More precisely, the mapping defined on the finite sums in $L_{2} (\{-1,1\}^{\infty}, X)$ by
$$P_{1} \Big( \sum_{A } x_{A} \varepsilon_{A} \Big) = \sum_{\vert A \vert =1} x_{A} \varepsilon_{A}$$ extends to bounded linear operator
$P_{1} : L_{2} (\{-1,1\}^{\infty}, X) \to L_{2} (\{-1,1\}^{\infty}, X)$.

If $X$ is $K$-convex, we can also define for each $m$ the projection $P_{m} : L_{2} (\{-1,1\}^{\infty}, X) \to L_{2} (\{-1,1\}^{\infty}, X)$, which on finite sums is given by
$P_{m} \Big( \sum_{A } x_{A} \varepsilon_{A} \Big) = \sum_{\vert A \vert =m} x_{A} \varepsilon_{A}$. By \cite[Theorem~2.1]{Pi82} or \cite[Theorem~13.16]{DiJaTo95}, there exists $K>1$ such that
\begin{equation} \label{estar}
\Vert P_{m} \Vert \leq K^{m}
\end{equation}
for every $m$. Also, a Banach space is $K$-convex if and only if it has nontrivial type (see e.g. \cite[Theorem~13.3]{DiJaTo95}).

\section{Type, cotype and Hausdorff-Young inequalities} \label{sec-tyc_pol}

In this section we present a polynomial reformulation of type/cotype and use it to prove Hausdorff-Young inequalities for Dirichlet series. We also provide a slightly stronger result for spaces enjoying uniform $\C$-convexity.
\subsection{A polynomial reformulation of type and cotype}

In \cite{CaDeSe16} the notion of \textit{hypercontractive homogeneous cotype} was introduced, as  an extension of the `usual' (or, to be more accurate, Rademacher) cotype. For a multi-index $\alpha = (\alpha_{1}, \ldots , \alpha_{n} , 0 , 0 , \ldots) \in \mathbb{N}_{0}^{(\mathbb{N})}$ we write $\vert \alpha \vert = \alpha_{1} + \cdots + \alpha_{n}$.
With this notation, a Banach space $X$ has \textit{hypercontractive homogeneous cotype} $q$ if there exists $C >0$ such that for every $m\in \mathbb N$ and every finite family  $(x_{\alpha} )_{\vert \alpha \vert =m}$ we have
\begin{equation}\label{eq0}
	\Big( \sum_{\vert \alpha \vert = m} \Vert x_{\alpha} \Vert^{q} \Big)^{\frac{1}{q}}
	\leq C^{m} \Big( \int_{\mathbb{T}^{n}} \Big\Vert  \sum_{\vert \alpha \vert = m}  x_{\alpha} z^{\alpha} \Big\Vert^{2} dz \Big)^{\frac{1}{2}} \,.
\end{equation}
Then,  different conditions were presented in \cite{CaDeSe16} which ensure that a Banach space $X$ enjoys this property. As a consequence of Theorem~\ref{cotimplicapolcot} (and Remark~\ref{expo}) below we see that, actually, every Banach space with  Rademacher cotype has hypercontractive homogeneous cotype.

\begin{theorem}
	\label{cotimplicapolcot}
	For a Banach space $X$ and $2\leq q<\infty$ the following statements are equivalent:
	\begin{enumerate}[label=(\alph*)]
		\item 	\label{proposition5-a}	$X$ has  cotype $q$;
		\item	\label{proposition5-b} there exists $C>0$ such that for every $m\in \mathbb N$ and every finite family  $(x_{\alpha} )_{\vert \alpha \vert \le m}$
\begin{equation} \label{eq3}
	\Big( \sum_{\vert \alpha \vert \le  m} \Vert x_{\alpha} \Vert^{q} \Big)^{\frac{1}{q}}
	\leq C^{m} \Big( \int_{\mathbb{T}^{n}} \Big\Vert  \sum_{\vert \alpha \vert \le m}  x_{\alpha} z^{\alpha} \Big\Vert^{q} dz \Big)^{\frac{1}{q}} \,;
\end{equation}
		\item 	\label{proposition5-c} there exists $C>0$ such that for every $m\in \mathbb N$ and every finite family $\{ x_{A} \colon  \vert A \vert \le m \} \subset X$ we have
		\begin{equation} \label{shankar}
		\Big( \sum_{A} \Vert x_{A} \Vert^{q} \Big)^{\frac{1}{q}} \leq C^{m} \Big( \E  \Big\Vert  \sum_{A}  x_{A} \varepsilon_{A} \Big\Vert^{q} \Big)^{\frac{1}{q}} \,.
		\end{equation}

	\end{enumerate}
\end{theorem}
The proof of this theorem is rather technical, so we postpone it  to Section~\ref{sec1}.
Although it was not considered in \cite{CaDeSe16}, for our purposes we also need a \emph{hypercontractive homogeneous type}. This,  again, turns out to be equivalent to the usual concept of Rademacher type, as follows from the next theorem.

\begin{theorem}
	\label{typimplicapoltyp}
	For a Banach space $X$ and $1\leq p\le 2$ the following statements are equivalent:
	\begin{enumerate}[label=(\alph*)]
		\item 	\label{propositionx-a}	$X$ has  type $p$;
		\item	\label{propositionx-b} there exists $C>0$ such that for every $m\in \mathbb N$ and every finite family  $(x_{\alpha} )_{\vert \alpha \vert \le m}$
\begin{equation}
	\label{eqx}
\Big( \int_{\mathbb{T}^{n}} \Big\Vert  \sum_{\vert \alpha \vert \le m}  x_{\alpha} z^{\alpha} \Big\Vert^{p} dz \Big)^{\frac{1}{p}}	
	\leq C^{m}  \Big( \sum_{\vert \alpha \vert \le  m} \Vert x_{\alpha} \Vert^{p} \Big)^{\frac{1}{p}}
 \,;
\end{equation}
		\item 	\label{propositionx-c} there exists $C>0$ such that for every $m\in\N$ and every finite family $\{ x_{A} \colon  \vert A \vert \le m \} \subset X$ we have
		\begin{equation} \label{shankarx}  \Big( \E  \Big\Vert  \sum_{A}  x_{A} \varepsilon_{A} \Big\Vert^{p} \Big)^{\frac{1}{p}}
		 \leq C^{m} \Big( \sum_{A} \Vert x_{A} \Vert^{p} \Big)^{\frac{1}{p}}
 \,.
		\end{equation}
	\end{enumerate}
\end{theorem}

\begin{remark}\label{expo}
	We want to stress the fact that the $L_{2}$-norm in \eqref{eq0}, the $L_q$-norm in \eqref{eq3} and the $L_{p}$-norm in \eqref{eqx} can be replaced by any other $L_r$ norm (in some cases with  different constants, but still of exponential growth in $m$). For homogeneous polynomials this is an immediate consequence of the polynomial Kahane inequality shown in \cite[Proposition~1.2]{CaDeSe16}: for $1\le s\le r<\infty$,
	if $\{ x_\alpha \colon \alpha \in \mathbb{N}_{0}^{(\mathbb{N})}, \,  |\alpha|=m \}\subset X$ is a family  with only finitely many nonzero elements, then
	\begin{equation} \label{proposition7}
	\Big\Vert  \sum  x_{\alpha} z^{\alpha} \Big\Vert_{L_{r}(\mathbb{T}^\infty,X)}
	\le \left(\frac{r}{s}\right)^\frac{m}{2}
	\Big\Vert  \sum  x_{\alpha} z^{\alpha} \Big\Vert_{L_{s}(\mathbb{T}^\infty,X)}
	\end{equation}
	These inequalities extend to general polynomials of degree $m$ proceeding as in \cite[Theorem~8.10]{DeSe19_libro}.
	
Once we know that \eqref{proposition7} holds for polynomials of degree $m$ (not necessarily homogeneous),  Lemma~\ref{lemma1} below gives an analogous inequality for Walsh polynomials: for any family $\{ x_{A} \colon A \subset \mathbb{N} , \, \vert A \vert \le m \} \subset X$ with only finitely many nonzero elements we have
	\begin{equation} \label{bonami}
	\Big( \E  \Big\Vert  \sum_{A}  x_{A} \varepsilon_{A} \Big\Vert^{r} \Big)^{\frac{1}{r}}
	\leq \left((1+\sqrt{2}) \sqrt{\frac{r}{s}}\right)^{m} \Big( \E  \Big\Vert  \sum_{A}  x_{A} \varepsilon_{A} \Big\Vert^{s} \Big)^{\frac{1}{s}} \,.
	\end{equation}
	This inequality should be compared to \cite[Corollary~5.5]{Pi16}, from which the homogeneous case of \eqref{bonami} constants can be deduced, with better constants. 
	As a consequence, the
	exponent~2 in the expectations in inequalities \eqref{shankar} and \eqref{shankarx} can be also replaced by any other exponent $r$.
\end{remark}

As we pointed out earlier, the proof of Theorem~\ref{cotimplicapolcot} is given in Section~\ref{sec1}, but let us sketch here the main ideas.
\paragraph{Structure of the proof.}
To begin with, let us recall that an $X$-valued polynomial of $n$ variables is a function $P: \mathbb{C}^{n} \to X$ given by a finite sum
\[
P(z) = \sum_{\alpha \in \Lambda \subseteq \mathbb{N}_{0}^{n}} x_{\alpha} z^{\alpha} \,,
\]
where $x_{\alpha} \in X$ for every $\alpha$. The degree of the polynomial is de maximum over $\Lambda$ of $\vert \alpha \vert = \alpha_{1} + \cdots + \alpha_{n}$. A polynomial is $m$-homogeneous if $\vert \alpha \vert = m$ for every $\alpha \in \Lambda$.
A quick thought shows that $X$ has cotype $q$ (see \eqref{defcotipo}) if and only if
\[
\Big( \sum_{\alpha} \Vert x_{\alpha} \Vert^{q} \Big)^{\frac{1}{q}} \leq C \Vert P \Vert_{L_{q}}
\]
for every polynomial of degree $1$ (note that the constant $C$ does not depend on the number of variables $n$). Also, \eqref{eq3} can be reformulated as
\begin{equation} \label{beira}
\Big( \sum_{\alpha} \Vert x_{\alpha} \Vert^{q} \Big)^{\frac{1}{q}} \leq C^{m} \Vert P \Vert_{L_{q}}
\end{equation}
for every polynomial of degree $m$ (here $C$ depends neither on $n$ nor on $m$). We begin the proof of Theorem~\ref{cotimplicapolcot} by showing that the inequality we aim at holds for a specific, easier to handle, class of polynomials: tetrahedral. These are polynomials where no power bigger than $1$ appears or, in other words, the monomials involved consist only of products of different variables. More precisely, a tetrahedral polynomial of $n$ variables is of the form
\[
\sum_{\alpha \in \{0 , 1\}^{n}} x_{\alpha} z^{\alpha} \,.
\]
Then, the first step towards the proof of Theorem~\ref{cotimplicapolcot} is to show in Lemma~\ref{cotawalsh} that \eqref{beira} (or, equivalently, \eqref{eq3}) holds for $m$-homogeneous tetrahedral polynomials. The second step is to show that the same inequality holds for every homogeneous polynomial. It is this second step (to pass from tetrahedral homogeneous to arbitrary homogeneous polynomials) that requires some work. In order to achieve this, we have to translate our results to the Walsh setting.

Note that, given $A\subseteq [n]$ we can define $\alpha = (\alpha_{i})_{i} \in \{0 , 1\}^{(\mathbb{N})}$ as $\alpha_{i} = 1$ if $i \in A$ and $0$ if $i \not\in A$. With this idea, to each finite set we can associate a tetrahedral multi-index  (and vice-versa). Denoting $z_{A} = \prod_{i \in A} z_{i}$ for each $A \subseteq [n]$ we may rewrite each tetrahedral polynomial as $\sum_{A\subseteq [n]} x_A z_A$. In this way, we have a straightforward identification between tetrahedral and Walsh polynomials:
\[
\sum_{\alpha \in \{0 , 1\}^n} x_{\alpha} z^{\alpha} = \sum_{A\subseteq [n]} x_A z_A \leftrightsquigarrow \sum_{A\subseteq [n]} x_A \varepsilon_A \,.
\]
In Lemma~\ref{lemma1} we relate the $L_{p}(\mathbb{T^{\infty}}, X)$-norm of the tetrahedral polynomial $\sum_{A\subseteq [n]} x_A z_A$ with the $L_{p}({\{-1,1\}^{\infty}}, X)$-norm of the Walsh polynomial $\sum_{A\subseteq [n]} x_A \varepsilon_A$ and in this way obtain that \eqref{shankar} holds for homogeneous Walsh polynomials. Then Lemma~\ref{lemma3} shows how to pass from homogeneous to arbitrary Walsh polynomials, completing the proof of \ref{proposition5-a} implies \ref{proposition5-c}. Finally, to deduce from \ref{proposition5-c} that \eqref{beira} holds for homogeneous polynomials goes through a convoluted description of a polynomial given in Lemma~\ref{lemma4}. To pass from homogeneous to arbitrary polynomials is rather stantard, and this finishes the proof.

\begin{remark}\label{loquevaleahora}
Following exactly the same arguments as in \cite[Theorem~5.3]{DeMaSc_12} (see also \cite[Proposition~25.29]{DeSe19_libro}) it can be shown that if $Y$ is a cotype $q$ space and $v : X \to Y$ is an $(r,1)$-summing operator, then there exists a constant $C>0$ so that
\[
\bigg( \sum_{\alpha} \Vert v(c_{\alpha}) \Vert_{Y}^{\frac{qrm}{q+(m-1)r}} \bigg)^{\frac{q+(m-1)r}{qrm}} \leq C^{m} \sup_{z \in \mathbb{D}^{n}} \Vert P(z) \Vert_{X}
\]
for every $X$-valued polynomial $P (z) = \sum_{\alpha} c_{\alpha} z^{\alpha}$ of $n$ variables of degree $m$. With this at hand, the estimates in  \cite[Theorem~1.6--(2) and Theorem~5.4--(2)]{DeMaSc_12} hold for Banach spaces with cotype $q$.

  Every $f \in H_{p} (\mathbb{T}^{\infty},X)$ defines a formal power series in infinitely many variables $\sum_{\alpha} \hat{f}(\alpha) z^{\alpha}$. The set of $z$s for which the power series of
every $f$ in $H_{p} (\mathbb{T}^{\infty},X)$ converges is called the set of monomial convergence:
\[
\mon H_{p} (\mathbb{T}^{\infty},X) = \Big\{ z \in \mathbb{C}^{\mathbb{N}} \colon
\sum_{\alpha} \Vert  \hat{f}(\alpha) z^{\alpha} \Vert_{X} < \infty \text{ for all } f \in H_{p} (\mathbb{T}^{\infty},X)
\Big\} \,.
\]
Then, the equivalence between cotype and polynomial cotype given Theorem~\ref{cotimplicapolcot} combined with \cite[(16)]{CaDeSe16} shows that, if we denote $\cot (X) = \inf \{q  \colon X \text{ has cotype } q \}$, then
\[
\ell_{\cot(X)'} \cap B_{c_{0}}  \subseteq \mon H_{p} (\mathbb{T}^{\infty},X) \subseteq \ell_{\cot(X)' + \varepsilon} \cap B_{c_{0}}
\]
for every $\varepsilon >0$. If $X$ attains its optimal cotype (that is, if $X$ has cotype $\cot (X)$), then \cite[Theorem~3.1]{CaDeSe16} gives
\begin{equation} \label{mon}
\mon H_{p} (\mathbb{T}^{\infty},X) = \ell_{\cot(X)'} \cap B_{c_{0}} \,.
\end{equation}
Also, let us recall that a sequence $b=(b_{n})_{n}$ is an $\ell_{1}$-multiplier of $\mathcal{H}_{p}(X)$ if $\sum_{n=1}^{\infty} \Vert a_{n} \Vert_{X} \vert b_{n} \vert < \infty$
for every $\sum a_{n} n^{-s}$ in $\mathcal{H}_{p}(X)$. As an immediate consequence of \eqref{mon} (see \cite[Theorem~4.3]{CaDeSe16}) we have that if $X$ has cotype $\cot(X)$, then a
multiplicative $b$ (that is, $b_{mn}=b_{n}b_{m}$ for every $m,n$) is an $\ell_{1}$-multiplier of $\mathcal{H}_{p}(X)$ if and only if $b \in \ell_{\cot(X)'}$.
\end{remark}

\subsection{Hausdorff-Young inequalities}

Let us note that Fourier cotype, as formulated in \eqref{eq-analytic-cotype}, implies \eqref{eq3} with universal constants, independent
of $m$.
Under the much weaker assumption of  cotype, the exponential dependence on $m$ (as $C^{m}$) in Theorem \ref{cotimplicapolcot} still allows us to carry estimations from the polynomial to the Dirichlet setting at a reasonable price. 
An analogous situation holds for Banach spaces with type~$p$. 
We obtain inequalities, not only for Dirichlet series, but also for functions defined on $\mathbb{T}^{\infty}$ or $\{-1,1\}^{\infty}$, as in Proposition~\ref{proposition3-cotipo} or \eqref{tamboril}. Comparing what we obtain with those inequalities, we gather that the $r$ factor in the following theorem is the price we pay for loosening the hypothesis of Fourier or Walsh to just Rademacher cotype. Let us recall that  the number of prime divisors of  $n \in \mathbb{N}$, counted with multiplicity is denoted by $\Omega (n)$.

\begin{theorem} \label{ponzio}
 For a Banach space $X$ and  $2\leq q <\infty$ the following statements are equivalent:
 \begin{enumerate}[label=(\alph*)]
 	\item \label{ponzio1} X has  cotype $q$;

 	\item \label{ponzio3} for some (every) $1\le p <\infty$, there exist constants $C\geq1$ and $0<r<1$ such that every vector-valued Dirichlet series $D = \sum a_{n} n^{-s} \in \mathcal{H}_p(X)$ satisfies
\[
\left(\sum_{n=1}^\infty r^{\Omega(n)}\|a_n\|^q\right)^{\frac{1}{q}} \leq C \|D\|_{\mathcal{H}_p(X)}.
\]

\item \label{ponzio2} for some (every) $1\le p <\infty$, there exist constants $C\geq1$ and $0<r<1$ such that every function $f\in H_p(\T^\infty,X)$ satisfies
 	\[
 	\Bigg(\sum_{\alpha\in \N_0^{(\N)}} r^{ |\alpha|} \|\widehat {f}(\alpha)\|^q\Bigg)^{\frac{1}{q}} \leq C \|f\|_{H_p(\T^\infty,X)};
 	\]
 	 \end{enumerate}
 In addition, the next statement~\ref{ponzio4} implies~\ref{ponzio1},\ref{ponzio2} and~\ref{ponzio3} and is equivalent to them whenever $X$ is $K$-convex:
 \begin{enumerate}[label=(\alph*)] \setcounter{enumi}{3}
 	\item\label{ponzio4} for some (every) $1<p<\infty$, there exist constants $C\geq1$ and $0<r<1$ such that every function $f\in L_p(\{-1,1\}^\infty,X)$ satisfies
 	\[
 	\Bigg(\sum_{\substack{A \subset \mathbb{N} \\ A \text{ finite}}} r^{ \vert A \vert} \|\widehat {f}(A)\|^q\Bigg)^{\frac{1}{q}} \leq C \|f\|_{L_p(\{-1,1\}^\infty,X)}.
 	\]
 \end{enumerate}
\end{theorem}
\begin{proof}
Observe that~\ref{ponzio2} and~\ref{ponzio3} are equivalent via Bohr's transform. The fact that~\ref{ponzio2}$\Rightarrow$\ref{ponzio1} follows by noticing that, given $x_{1}, \ldots , x_{N} \in X$ and defining $P(z) = \sum_{n=1}^{N} x_{n} z_{n}$, the sum at the left-hand side becomes $r^{1/q} \big(\sum  \Vert x_{n} \Vert^{q} \big)^{1/q}$. The same argument proves that~\ref{ponzio4}$\Rightarrow$\ref{ponzio1} invoking Theorem~\ref{cotimplicapolcot}.

Next we see that~\ref{ponzio1}$\Rightarrow$\ref{ponzio2}. Set $f\in H_p(\T^\infty,X)$ and for every $m\in \N$ let $f_m$ be its $m$-homogeneous projection (see \cite[Proposition~2.5]{CaDeSe14}). By Theorem~\ref{cotimplicapolcot}, there is a constant $c\ge 1$ such that for every (finite) $m$-homogeneous polynomial $P=\sum x_\alpha z^\alpha$ we have
\[
\Big( \sum_{\vert \alpha \vert = m} \Vert x_{\alpha} \Vert^{q} \Big)^{\frac{1}{q}} \leq c^{m} \|P\|_p.
\]
Since polynomials are dense in $H_p(\T^\infty,X)$ and the $m$-homogeneous projection is a contraction, a straightforward density argument yields
\begin{equation} \label{pitymartinez}
\Big( \sum_{\vert \alpha \vert = m} \Vert \widehat{f}(\alpha) \Vert^{q} \Big)^{\frac{1}{q}} \leq c^{m} \|f_m\|_p\le c^{m} \|f\|_p.
\end{equation}
Taking $r<1/c^q$ we get
\begin{align*}
\Big(\sum_{\alpha\in \N_0^{(\N)}} r^{ |\alpha|} \|\widehat {f}(\alpha)\|^q\Big)^{\frac{1}{q}} =
\Big(\sum_{m=1}^{\infty} r^{m} \sum_{|\alpha|=m} \|\widehat {f}(\alpha)\|^q\Big)^{\frac{1}{q}}
\le \Big(\sum_{m=1}^{\infty} (rc^q)^{m} \Big)^{\frac{1}{q}}\|f\|_p
\le C \|f\|_p,
\end{align*}
which completes the argument.

We finally show that~\ref{ponzio1}$\Rightarrow$\ref{ponzio4} for $K$-convex spaces . First assume that $p=2$. In this case \eqref{estar} gives a constant $K$ so that
\begin{equation} \label{esta}
\Vert f_{m} \Vert_{2} \leq K^{m} \Vert f \Vert_{2}
\end{equation}
for  every $f \in L_2(\{-1,1\}^\infty,X)$. This enables us to proceed exactly as in \eqref{pitymartinez} to get the desired result. For the general case when $1<p<\infty$, it only remains to show that an  inequality analogous to \eqref{esta} holds. On the one hand, if $2\leq p <\infty$, using \eqref{bonami} we get
\[
\Vert f_{m} \Vert_{p} \leq C^m\Vert f_{m} \Vert_{2} \leq (CK)^m\Vert f \Vert_{2}\leq (CK)^m\Vert f \Vert_{p},
\]
for some constant $C>0$.

On the other hand,  it is a well-known fact that if $X$ is $K$-convex, so is $X^*$ (see for example \cite[Corollary~13.7 and Theorem~13.15]{DiJaTo95}). Therefore, for  $1<p\le 2$,
\begin{multline*}
\Vert f_{m} \Vert_{p}= \sup_{\substack{g\in L_{p'}(X^*)\\\|g\|_{p'}=1}} \E[g(\varepsilon)(f_m(\varepsilon))] =
\sup_{\substack{g\in L_{p'}(X^*)\\\|g\|_{p'}=1}} \E[g_m(\varepsilon)(f(\varepsilon))] \\
\le \sup_{\substack{g\in L_{p'}(X^*)\\\|g\|_{p'}=1}} \|g_m\|_{p'} \|f\|_p
\le \sup_{\substack{g\in L_{p'}(X^*)\\ \|g\|_{p'}=1}} \widetilde{K}^m \|g\|_{p'} \|f\|_p \le \widetilde{K}^m \|f\|_p,
\end{multline*}
for some constant $\widetilde{K}>0$.
\end{proof}

As a consequence of Theorem \ref{ponzio} we obtain the following result mentioned in the previous section.
\begin{corollary}\label{existente}
	Let $X$ be a Banach space with cotype $q$ and set $1\leq p \leq \infty$. For every $\delta>0$ there is a constant $C\geq1$ such that every $D = \sum a_{n} n^{-s} \in \mathcal{H}_p(X)$ satisfies
	\begin{align*}
	%\label{fractalosa2}
	\Big(\sum_{n=1}^\infty \frac{\|a_n\|^q}{n^\delta}\Big)^{1/q} \leq C \|D\|_{\mathcal{H}_{p}(X)}.
	\end{align*}
\end{corollary}
\begin{proof}
	From Theorem \ref{ponzio} we know that there exist constants $C\geq1$ and $0<r<1$ such that every $D = \sum a_{n} n^{-s} \in \mathcal{H}_p(X)$ satisfies
	\[
	\left(\sum_{n=1}^\infty r^{\Omega(n)}\|a_n\|^q\right)^{\frac{1}{q}} \leq C \|D\|_{\mathcal{H}_p(X)}.
	\]
	Fix $\delta>0$ and let $k\in \N$ be such that $1/p_k^\delta\leq r$ where $p_k$ denotes the $k$-th prime number. Notice that if $p_k=2$ we are done since we would get $1/n^\delta\leq r^{\Omega(n)}$. However if $p_k>2$ we must deal with the first $k$ primes where the estimation by $r$ fails. This procedure is analogous to 
\cite[Lemma~2]{DeGaMaPG08}
so we only sketch the proof. Fix $D = \sum a_{n} n^{-s} \in \mathcal{H}_p(X)$ and consider $f=\mathfrak{B}_X^{-1}(D)\in
	H_p(\T^\infty,X)$. For $\alpha_1,\ldots,\alpha_k\in \N_0$ define
	\[f_{\alpha_1,\ldots,\alpha_k}(z)
	=\int_{\T^k} f(\omega_1,\ldots,\omega_k,z_{k+1},z_{k+2},\ldots) \omega_1^{-\alpha_1}\ldots\omega_k^{-\alpha_k} d\omega.\]
	An easy computation shows that $f_{\alpha_1,\ldots,\alpha_k}\in
	H_p(\T^\infty,X)$ and $\|f_{\alpha_1,\ldots,\alpha_k}\|_p\le \|f\|_p$. Moreover, for $\beta\in \N_0^{(\N)}$ we have that
	\[\widehat{f}_{\alpha_1,\ldots,\alpha_k}(\beta)=\begin{cases}
	\widehat{f}(\alpha_1,\ldots,\alpha_k,\beta_{k+1},\beta_{k+2},\ldots) \quad &\text{if }\beta=(0,\ldots,0,\beta_{k+1},\beta_{k+2},\ldots)
	\\ 0 \quad &\text{otherwise.}
	\end{cases}\]
	Applying \ref{ponzio2} of Theorem \ref{ponzio} to $f_{\alpha_1,\ldots,\alpha_k}$ we get
	\[\Big(\sum_{\beta\in \N_0^{(\N)}} r^{ |\beta|} \|\widehat {f}(\alpha_1,\ldots,\alpha_k,\beta)\|^q\Big)^{\frac{1}{q}} \leq C \|f\|_p.\]
	Therefore we deduce
	\begin{multline*}
	\sum_{n=1}^\infty \frac{\|a_n\|^q}{n^\delta}
	\leq\sum_{\alpha_1,\ldots,\alpha_k\geq 0}p_1^{-\alpha_1\delta}\ldots p_k^{-\alpha_k\delta}
	\sum_{\beta\in \N_0^{(\N)}} r^{|\beta|}\|\widehat {f}(\alpha_1,\ldots,\alpha_k,\beta)\|^q
	\\
	\leq C^q  \sum_{\alpha_1,\ldots,\alpha_k\geq 0}p_1^{-\alpha_1\delta}\ldots p_k^{-\alpha_k\delta} \|f\|_p^q
	=C^q  \Big(\prod_{j=1}^k \frac{1}{1-1/p_j^\delta}\Big) \|f\|_p^q,
	%\\ &\leq C^q \left(\frac{2^\delta}{2^\delta-1}\right)^k \|f\|_p^q,
	\end{multline*}
	which completes the proof.
\end{proof}

We now turn our attention to spaces with nontrivial type, and get an analogous result (compare it also with Proposition~\ref{proposition3} and \eqref{gaita}). The proof follows essentially the same lines as that of Theorem~\ref{ponzio} (in fact, it is slightly simpler) so we omit it.

\begin{theorem}\label{maidana}
	 For a Banach space $X$ and  for $1\leq p \leq 2$ the following statements are equivalent:
	 \begin{enumerate}[label=(\alph*)]
	 	\item X has type $p$;
		
		\item for some (every) $1\leq q <\infty$ there exist constants $R,C\geq1$ such that every $X$-valued Dirichlet series $D= \sum a_{n} n^{-s}$ satisfies
	 	\[
		\|D\|_{\mathcal{H}_q(X)} \leq C \left(\sum_{n=1}^\infty R^{\Omega(n)}\|a_n\|^p\right)^{\frac{1}{p}};
		\]
	 	
	 	\item for some (every) $1\leq q <\infty$  there exist constants $C, R \geq1$ and such that every function $f\in H_1(\T^\infty,X)$  satisfies
	 	\[
	 	\|f \|_{H_q(\T^\infty,X)} \leq C \Bigg(\sum_{\alpha\in \N_0^{(\N)}} R^{|\alpha|} \|\hat f(\alpha)\|^p\Bigg)^{\frac{1}{p}};
	 	\]
	 	
		\item for some (every) $1\leq q <\infty$ there exist constants $C, R\geq1$ such that every function $f\in L_1(\{-1,1\}^\infty,X)$ satisfies
	 	\[
	 	\|f \|_{L_q(\{-1,1\}^\infty,X)} \leq C \Bigg(\sum_{\substack{A \subset \mathbb{N} \\ A \text{ finite}}} R^{ \vert A \vert}\|\hat f(A)\|^p\Bigg)^{\frac{1}{p}}.
	 	\]
	 \end{enumerate}
\end{theorem}

The preceding inequalities should be understood as follows: if the sum at the right-hand side is finite, then the Dirichlet series (or the function) belongs to the corresponding space and its norm is controlled by the sum. But if the sum does not converge, then nothing can be said about the series or the function.

\subsection{Uniform $\mathbb{C}$-convexity}

A Banach space $X$ is $q$-uniformly $\mathbb{C}$-convex \cite{Gl75} (for $q \geq 2$) if there exists $\lambda > 0$ such that
\[
\big( \Vert x \Vert^{q} + \lambda \Vert y \Vert^{q} \big)^{1/q} \leq
\max_{z \in \mathbb{T}} \Vert x + zy \Vert,
%\end{equation}
\]
for all $x,y \in X$ and $q$-uniformly $PL$-convex (see \cite{DaGaTJ84} or \cite[Chapter~11]{Pi16}) if
\begin{equation} \label{1890}
\Vert x \Vert^{q} + \lambda \Vert y \Vert^{q} \leq
\int_{\mathbb{T}} \Vert x + zy \Vert^{q} dz,
\end{equation}
for all $x,y \in X$. In fact these two concepts are equivalent (see \cite{Pa91}) and provide an analytic version of the more familiar geometric property known as $q$-uniform convexity. A Banach space $X$ is $q$-uniformly convex (for $q \geq 2$) if there exists $\lambda > 0$ such that
\[
 \Vert x \Vert^{q} + \lambda \Vert y \Vert^{q} \leq
\E\Vert x + \varepsilon y \Vert^q.
%\end{equation}
\]
It is easy to check that $q$-uniform convexity implies $q$-uniform $PL$-convexity.

In \cite[Proposition~2.1]{BlPa03} it is proven that $q$-uniform $\mathbb{C}$-convexity is equivalent to either of the following conditions:
\begin{enumerate}[label=(\alph*)]
	\item there exists $\lambda>0$ such that for every analytic function $f:\mathbb{D}\rightarrow X$ we have
		\begin{equation} \label{quilmescero}
		\Vert f(0) \Vert^{q} + \lambda \Vert f'(0) \Vert^{q}
		\leq  \sup_{|z|<1} \Vert f(z) \Vert^{q}.
		\end{equation}
	\item there exists $\lambda>0$ such that for every analytic function $f:\mathbb{D}\rightarrow X$ we have
		\begin{equation} \label{quilmes}
		\Vert f(0) \Vert^{q} + \lambda \Vert f'(0) \Vert^{q}
		\leq  \sup_{0<r<1}\,\int_{\mathbb{T}} \Vert f(rz) \Vert^{q} dz.
		\end{equation}
\end{enumerate}
Let us note that for every such function the mapping $r \in [0,1[ \rightsquigarrow \Vert  f (r \punkt) \Vert_{H_{q}(\mathbb{T},X)}$ is increasing and, then, the supremum at the right-hand side of \eqref{quilmes} is in fact a limit as $r \to 1^{-}$. With this, if $f : \mathbb{C} \to X$ is entire, then
\begin{equation} \label{stella}
\Vert f(0) \Vert^{q} + \lambda \Vert f'(0) \Vert^{q}
\leq
%\sup_{0<r<1}\,\int_{\mathbb{T}} \Vert f(rz) \Vert^{q} dz=
 \int_{\mathbb{T}}\Vert f(z) \Vert^{q} dz  \,.
\end{equation}
Since taking $f(z)= x+zy$ for given $x$ and $y$ gives \eqref{1890}, the equivalence with $q$-uniform $\mathbb{C}$-convexity is mantained.

Using \eqref{quilmescero} Blasco proved in \cite[Theorem~2.4]{Bl17} that $q$-uniformly $\mathbb{C}$-convex spaces have positive $q$-Bohr radius. That is, there exists $\rho>0$ such that
\begin{equation}\label{blas}
\Big( \sum_{n=0}^{\infty}  \Vert x_n  \Vert^{q} \rho^{qn} \Big)^{1/q}
\leq  \sup_{|z|<1} \|f(z)\|,
\end{equation}
for every analytic function $f=\sum_n x_n z^n$ on $\mathbb{D}$. Replacing \eqref{quilmescero} by \eqref{stella} in his argument we deduce that for $q$-uniformly $\mathbb{C}$-convex spaces there exists $\rho>0$ such that
\begin{equation}
\label{caso1}
\Big( \sum_{n=0}^{\infty} \Vert x_n  \Vert^{q} \rho^{qn} \Big)^{1/q}
 \le \Big(\int_{\mathbb{T}}\Vert f(z) \Vert^{q} dz \Big)^{1/q},
\end{equation}
for every entire function $f=\sum_n x_n z^n$. The following theorem extends this fact to several variables.

\begin{theorem}\label{isenbeck}
Let $X$ be a  $q$-uniformly $\mathbb{C}$-convex Banach space. Then there exists $\rho>0$ such that for every $n$ and every polynomial $P=\sum x_\alpha z^\alpha$ of $n$ variables with values in $X$ we have
\[
\Big( \sum_{\alpha} \Vert x_{\alpha}  \Vert^{q} \rho^{\vert \alpha \vert q} \Big)^{1/q}
\leq \Big(\int_{\mathbb{T}^{n}}\Vert P(z) \Vert^{q} dz \Big)^{1/q} \,.
\]
\end{theorem}
\begin{proof}
	We proceed by induction on $n$, the number of variables. The case $n=1$ follows from \eqref{caso1}.

Suppose now that the result holds for $n-1$ and take some polynomial
\[
P(z) = \sum_{\alpha \in F} x_{\alpha} z^{\alpha} \,,
\]
for $z \in \mathbb{C}^{n}$ (where $F \subseteq \mathbb{N}_{0}^{n}$ is finite). Then we can write
\[
\sum_{\alpha} \Vert x_{\alpha}  \Vert^{q} \rho^{\vert \alpha \vert q}
= \sum_{k=0}^{N} \rho^{qk} \sum_{\substack {\alpha \in F \\ \alpha_{n}=k}}   \Vert x_{\alpha}  \Vert^{q} \rho^{(\vert \alpha \vert - \alpha_{n}) q} \,.
\]
Applying the inductive hypothesis to each polynomial
\[
z \in \mathbb{C}^{n-1} \rightsquigarrow \sum_{\substack {\alpha \in F \\ \alpha_{n}=k}}   x_{\alpha} z_{1}^{\alpha_{1}} \cdots z_{n-1}^{\alpha_{n-1}},
\]
 we have
\begin{multline*}
\sum_{\alpha} \Vert x_{\alpha}  \Vert^{q} \rho^{\vert \alpha \vert q}
\leq \sum_{k=0}^{N} \rho^{qk} \int_{\mathbb{T}^{n-1}} \Big\Vert \sum_{\substack {\alpha \in F \\ \alpha_{n}=k}}   c_{\alpha} z_{1}^{\alpha_{1}} \cdots z_{n-1}^{\alpha_{n-1}} \Big\Vert^{q} d(z_{1}, \ldots , z_{n-1}) \\
= \int_{\mathbb{T}^{n-1}}\sum_{k=0}^{N}   \Big\Vert \sum_{\substack {\alpha \in F \\ \alpha_{n}=k}}   c_{\alpha} z_{1}^{\alpha_{1}} \cdots z_{n-1}^{\alpha_{n-1}} \Big\Vert^{q} \rho^{qk} d(z_{1}, \ldots , z_{n-1}) \,.
\end{multline*}
Finally, for each fixed $(z_{1}, \ldots , z_{n-1}) \in \mathbb{T}^{n-1}$ we may consider the polynomial $\mathbb{C} \to X$ given by
\[
z \rightsquigarrow \sum_{k=0}^{n}
\Big( \sum_{\substack {\alpha \in F \\ \alpha_{n}=k}}   c_{\alpha} z_{1}^{\alpha_{1}} \cdots z_{n-1}^{\alpha_{n-1}} \Big) z^{k}
\]
and then use the case $n=1$ of the induction to conclude
\begin{multline*}
\sum_{k=0}^{N}   \Big\Vert \sum_{\substack {\alpha \in F \\ \alpha_{n}=k}}   x_{\alpha} z_{1}^{\alpha_{1}} \cdots z_{n-1}^{\alpha_{n-1}} \Big\Vert^{q} \rho^{qk} \\
\leq \int_{\mathbb{T}} \Big\Vert \sum_{k=0}^{N}
\Big( \sum_{\substack {\alpha \in F \\ \alpha_{n}=k}}   x_{\alpha} z_{1}^{\alpha_{1}} \cdots z_{n-1}^{\alpha_{n-1}} \Big) z_{n}^{k} \Big\Vert^{q} dz_{n}
= \int_{\mathbb{T}}\Big\Vert \sum_{\alpha \in F}   x_{\alpha} z_{1}^{\alpha_{1}} \cdots z_{n-1}^{\alpha_{n-1}}z_{n}^{\alpha_{n}} \Big\Vert^{q} dz_{n} \,.
\end{multline*}
Fubini's theorem completes the proof.
\end{proof}

Let us note that Theorem~\ref{isenbeck} can be reformulated as
\begin{equation}\label{radio}
\Big( \sum_{n \leq x} \Vert a_{n} \Vert^{q} \rho^{q\Omega (n)} \Big)^{\frac{1}{q}}
\leq  \Big\Vert \sum_{n \leq x} a_{n} n^{-s} \Big\Vert_{\mathcal{H}_{q}(X)}
\leq  \Big\Vert \sum_{n \leq x} a_{n} n^{-s} \Big\Vert_{\mathcal{H}_{\infty}(X)} \,,
\end{equation}
for every Dirichlet polynomial. This gives a version of the $q$-Bohr radius for vector-valued Dirichlet series (although for better constants one should proceed to the multivariate setting directly from \eqref{blas}). Also, the first inequality in \eqref{radio} gives the equivalence~\ref{ponzio3} of Theorem~\ref{ponzio} with constant $C=1$ (taking $r=\rho^q$ and $p=q$). Hence, for $q$-uniformly $\mathbb C$-convex Banach spaces we have better Hausdorff-Young inequalities than those for general spaces with cotype $q$.

\section{Proof of Theorem~\ref{cotimplicapolcot}} \label{sec1}

We face now the proof to Theorem~\ref{cotimplicapolcot}. Let us recall that we are aiming at inequalities like \eqref{eq3} for every polynomial of $n$ variables of degree $m$. We begin by showing that such an inequality
indeed holds for homogeneous tetrahedral polynomials. But before we get into that let us note that
given a vector space $V$, a family $\{ v_{A} \colon A \subseteq [n], \, \vert A \vert =m \}\subseteq V$ (where $n,m \in \mathbb{N}$) and $k \in \mathbb{N}$ we have 
\begin{equation}\label{idcomb}
\begin{split}
	\sum_{\substack{B\subseteq[n] \\ |B|=k}}
	&\sum_{\substack{A_1\subseteq B \\ |A_1|=1}}
	\sum_{\substack{A_2\subseteq B^c \\ |A_2|=m-1}} v_{A_1\cup A_2}
	= %\frac{1}{\binom{n}{k}}
	\sum_{\substack{B\subseteq[n] \\ |B|=k}}
	\sum_{\substack{A\subseteq [n] \\ |A|=m \\|A\cap B|=1}} v_A
	= %\frac{1}{\binom{n}{k}}
	\sum_{\substack{A\subseteq[n] \\ |A|=m}}
	\sum_{\substack{B\subseteq [n] \\ |B|=k \\ |A\cap B|=1}} v_A
	\\ &= %\frac{1}{\binom{n}{k}}
	\sum_{\substack{A\subseteq[n] \\ |A|=m}}
	\big|\{B\subseteq [n] : \ |B|=k, \ |A\cap B|=1\}\big| v_A
	= %\frac{m \binom{n-m}{k-1}}{\binom{n}{k}}
	m \binom{n-m}{k-1} \sum_{\substack{A\subseteq [n] \\ |A|=m}} v_A.
\end{split}
\end{equation}

This is essentially equations (3.13) through (3.16) from \cite{RzWo_19}.
Once we have this we can prove that the inequality we aim at holds for homogeneous tetrahedral polynomials.

\begin{lemma}\label{cotawalsh}
Let $X$ be a Banach space $X$ of cotype $2\leq q<\infty$. Then for every $m,n\in \N$ and every family $\{ x_{A} \colon A \subseteq [n], \, \vert A \vert =m \} \subseteq X$ we have
\begin{equation}\label{cotwalsh}
	\Big( \sum_{A} \Vert x_{A} \Vert^{q} \Big)^{\frac{1}{q}}
\leq (4^{1/q} C_{q}(X)) ^{m} \bigg( \int_{\T^n}  \Big\Vert  \sum_{A}  x_{A} z_{A} \Big\Vert^{q}  dz \bigg)^{\frac{1}{q}} \,.
\end{equation}
\end{lemma}
\begin{proof}
	We prove this by induction on $m$. The case $m=1$ is trivial just comparing \eqref{defcotipo} with \eqref{cotwalsh} for $m=1$.
	
For the inductive step, let $\{ x_{A} \colon A \subseteq [n], \, \vert A \vert =m \}$ be a family of vectors in $X$. Taking $n$ larger if necessary, we may assume $n=km$ for some $k\in \N$. Note that, in this case, $\binom{n-m}{k-1} = \binom{(k-1)m}{k-1}$ and a straightforward application of Stirling's formula yields
\[
	\frac{1}{2} \leq \frac{\binom{n}{k}}{m \binom{n-m}{k-1}} \leq 4.
\]
Using \eqref{idcomb} for $v_A=\|x_A\|^q$ we get
	\[
	\sum_{\substack{A\subseteq [n] \\ |A|=m}} \|x_A\|^q
	= \frac{1}{m \binom{n-m}{k-1}}
	\sum_{\substack{B\subseteq[n] \\ |B|=k}}
	\sum_{\substack{A_1\subseteq B \\ |A_1|=1}}
	\sum_{\substack{A_2\subseteq B^c \\ |A_2|=m-1}} \|x_{A_1\cup A_2}\|^q
	\leq \frac{4}{\binom{n}{k}}
	\sum_{\substack{B\subseteq[n] \\ |B|=k}}
	\sum_{\substack{A_1\subseteq B \\ |A_1|=1}}
	\sum_{\substack{A_2\subseteq B^c \\ |A_2|=m-1}} \|x_{A_1\cup A_2}\|^q.
	\]
	For a fixed $A_1$ we can apply the inductive hypothesis to the family $\{ x_{A_1\cup A_2} \colon A_2 \subseteq B^c, \, \vert A_2 \vert =m-1 \}$. Let $\T^{B^c}$ denote $|B^c|$ copies of the torus indexed in $B^c$. We get
	\begin{align*}
	\sum_{\substack{A\subseteq [n] \\ |A|=m}} \|x_A\|^q
	&\leq \frac{4^{m} C_{q}(X)^{q(m-1)}}{\binom{n}{k}}
	\sum_{\substack{B\subseteq[n] \\ |B|=k}}
	\sum_{\substack{A_1\subseteq B \\ |A_1|=1}}
	\int_{\T^{B^c}}  \Big\|\sum_{\substack{A_2\subseteq B^c \\ |A_2|=m-1}}
	x_{A_1\cup A_2} z_{A_2}\Big\|^{q} dz
	\\ &\leq \frac{4^{m} C_{q}(X)^{q(m-1)}}{\binom{n}{k}}
	\sum_{\substack{B\subseteq[n] \\ |B|=k}}
	\int_{\T^{B^c}}  \sum_{\substack{A_1\subseteq B \\ |A_1|=1}}
	\Big\|\sum_{\substack{A_2\subseteq B^c \\ |A_2|=m-1}}
	x_{A_1\cup A_2} z_{A_2}\Big\|^{q} dz.
	\end{align*}
	Notice that we are integrating over variables $z_i$ whose index $i$ always lies in $B^c$, while $A_1$ is always included in $B$. In some sense, the variables $z_{A_1}$ remain unused. So, by the cotype inequality \eqref{defcotipo} we obtain
	\begin{equation}\label{ec22}
	\sum_{\substack{A\subseteq [n] \\ |A|=m}} \|x_A\|^q
	%	&\leq \frac{C_2 C_q^q C^{q(m-1)}}{\binom{n}{k}}
	%	\sum_{\substack{B\subseteq[n] \\ |B|=k}}
	% \int_{\T^{B^c}}  \int_{\T^{B}} \Big\|\sum_{\substack{A_1\subseteq B \\ |A_1|=1}}
	%	\sum_{\substack{A_2\subseteq B^c \\ |A_2|=m-1}}
	%	x_{A_1\cup A_2} z_{A_1}z_{A_2}\Big\|^q dz \notag
	\leq\frac{4^{m} C_q(X)^{qm}}{\binom{n}{k}}
	\sum_{\substack{B\subseteq[n] \\ |B|=k}}
	\int_{\T^n} \Big\|\sum_{\substack{A_1\subseteq B \\ |A_1|=1}}
	\sum_{\substack{A_2\subseteq B^c \\ |A_2|=m-1}}
	x_{A_1\cup A_2} z_{A_1} z_{A_2}\Big\|^{q} dz.
	\end{equation}
	Regarding the last expression, observe that
	\begin{equation}\label{ecproy}
	\begin{split}
	\int_{\T} \Big(\sum_{j=0}^{\min(k,m)} & \sum_{\substack{A_1\subseteq B \\ |A_1|=j}}
	\sum_{\substack{A_2\subseteq B^c \\ |A_2|=m-j}}
	x_{A_1\cup A_2}  \omega^{|A_1|}z_{A_1} z_{A_2}\Big) \overline{\omega} d\omega \\
	& = \sum_{j=0}^{\min(k,m)} \int_{\T} \omega^{j-1} d\omega
	\sum_{\substack{A_1\subseteq B \\ |A_1|=j}}
	\sum_{\substack{A_2\subseteq B^c \\ |A_2|=m-j}}
	x_{A_1\cup A_2} z_{A_1} z_{A_2}
	=\sum_{\substack{A_1\subseteq B \\ |A_1|=1}}
	\sum_{\substack{A_2\subseteq B^c \\ |A_2|=m-1}}
	x_{A_1\cup A_2} z_{A_1} z_{A_2}.
	\end{split}
	\end{equation}
	Using \eqref{ecproy} in \eqref{ec22} and applying Jensen's inequality we have
	\begin{align*}
	\sum_{\substack{A\subseteq [n] \\ |A|=m}} \|x_A\|^q
	&\leq\frac{4^{m} C_{q}(X)^{qm}}{\binom{n}{k}}
	\sum_{\substack{B\subseteq[n] \\ |B|=k}}
	\int_{\T} \int_{\T^n} \Big\|\sum_{j=0}^{\min(k,m)}
	\sum_{\substack{A_1\subseteq B \\ |A_1|=j}}
	\sum_{\substack{A_2\subseteq B^c \\ |A_2|=m-j}}
	x_{A_1\cup A_2}  \omega^{|A_1|} z_{A_1} z_{A_2}\Big\|^q dz d\omega.
	\end{align*}
	Finally, by rotation invariance $\omega^{|A_1|} z_{A_1}$ may be replaced by $z_{A_1}$. We get
	\begin{align*}
	\sum_{\substack{A\subseteq [n] \\ |A|=m}} \|x_A\|^q
	&\leq\frac{4^{m} C_{q}(X)^{qm}}{\binom{n}{k}}
	\sum_{\substack{B\subseteq[n] \\ |B|=k}}
	\int_{\T^n} \Big\|\sum_{j=0}^{\min(k,m)}
	\sum_{\substack{A_1\subseteq B \\ |A_1|=j}}
	\sum_{\substack{A_2\subseteq B^c \\ |A_2|=m-j}}
	x_{A_1\cup A_2}  z_{A_1} z_{A_2}\Big\|^q dz
	\\ &=\frac{4^{m} C_{q}(X)^{qm}}{\binom{n}{k}}
	\sum_{\substack{B\subseteq[n] \\ |B|=k}}
	\int_{\T^n} \Big\|
	\sum_{\substack{A\subseteq [n] \\ |A|=m}}
	x_{A}  z_{A}\Big\|^q dz
	=4^{m} C_{q}(X)^{qm}
	\int_{\T^n} \Big\|
	\sum_{\substack{A\subseteq [n] \\ |A|=m}}
	x_{A}  z_{A}\Big\|^q dz. \qedhere
	\end{align*}
\end{proof}

In order to deal with Walsh polynomials in Theorem~\ref{cotimplicapolcot}~\ref{proposition5-c} we need two lemmas. The first one shows that tetrahedral Steinhaus polynomials and their Walsh counterparts have equivalent norms up to exponential constants. The argument translates estimates from the scalar to the Banach setting applying a theorem of Pe{\l}czy\'nski. This result can also be proven using \cite[Proposition 6.3.1]{KwWo92} and checking the hypothesis by hand.

\begin{lemma} \label{lemma1}
 Let $X$ be a Banach space and set $1\leq q <\infty$. For every tetrahedral polynomial $P$ of degree $m$ and $n$ variables we have
 \begin{equation}\label{sw}
   (1+\sqrt{2})^{-m} \left(\mathbb{E}\|P(\varepsilon)\|_X^q\right)^{1/q}
   \leq	\Big( \int_{\mathbb{T}^{n}} \Vert  P(z) \Vert_X^{q} dz \Big)^{1/q}
   \leq (1+\sqrt{2})^{m} \left(\mathbb{E}\|P(\varepsilon)\|_X^q\right)^{1/q}.
\end{equation}
\end{lemma}

\begin{proof}
	In \cite[p.~2764]{kli95} it is shown that for every polinomial $Q:\C^n\rightarrow \C$ of degree $m$, we have
	\[\sup_{z\in\T^n}|Q(z)|\leq (1+\sqrt 2)^m \sup_{x\in[-1,1]^n} |Q(x)|.\]
	If we assume $Q$ to be tetrahedral, we observe as in \cite{DeMaPe19} that
	\[\sup_{x\in[-1,1]^n} |Q(x)|= \sup_{\varepsilon\in\{-1,1\}^n} |Q(\varepsilon)|,\]
	since $Q$ is affine in every coordinate. Thus,
	\begin{equation*}
	\sup_{\varepsilon\in\{-1,1\}^n}|Q(\varepsilon)|\le \sup_{z\in\T^n}|Q(z)|\leq (1+\sqrt 2)^m \sup_{\varepsilon\in\{-1,1\}^n}|Q(\varepsilon)|.
	\end{equation*}
	
Equivalently, for every finite choice of scalars $\{c_A\}_{|A|\le m}\subseteq \C$ we have
\begin{equation}
\label{eqinfty}
\sup_{\varepsilon\in\{-1,1\}^n}\Big|\sum_{|A|\le m}c_A \varepsilon_A\Big|
 \leq \sup_{z\in\T^n}\Big| \sum_{|A|\le m}c_A z_A\Big|
 \leq (1+\sqrt{2})^{m} \sup_{\varepsilon\in\{-1,1\}^n}\Big|\sum_{|A|\le m}c_A \varepsilon_A\Big|,
\end{equation}
where for simplicity we also used Walsh notation for the variable $z$.
Consider the sets of characters
$\{\varepsilon_A\}_{|A|\le m}$ and
$\{z_A\}_{|A|\le m}$ of the compact abelian groups $\{-1,1\}^n$ and $\T^n$ respectively. Since these sets satisfy \eqref{eqinfty}, the conditions of
\cite[Theorem~1]{Pel88} are met. So we get
	\begin{multline*}
	(1+\sqrt{2})^{-m} \Big\|\sum_{|A|\le m}x_A \varepsilon_A\Big\|_{L^q(\{-1,1\}^n,X)}
	 \leq \Big\| \sum_{|A|\le m}x_A z_A \Big\|_{L^q(\T^n,X)}
	\\ \leq (1+\sqrt{2})^{m} \Big\|\sum_{|A|\le m}x_A \varepsilon_A\Big\|_{L^q(\{-1,1\}^n,X)},
	\end{multline*}
	 for every choice of vectors $\{x_A\}_{|A|\le m}\subseteq X$.
	This concludes the proof since it is equivalent to \eqref{sw}.
\end{proof}

The following lemma estimates the norm of the homogeneous projection of a Walsh polynomial and can be found in \cite[Lemma~2]{Kw87} (see also \cite[Lemma 3.2.4]{DPGi99}). A proof is included since the constant  is not explicitly computed there, and we need it to grow exponentially on the degree of the polynomial (i.e., to be of the form $B^m$ for some $B>0$).

\begin{lemma} \label{lemma3}
	For every Banach space $X$ there is a constant $B>0$ such that for every $1\leq q<\infty$ and every $X$-valued Walsh polynomial $P$ of degree $m$, its $k$-homogeneous projection $P_k$ satisfies
	\[
	\big(\mathbb{E}\|P_k(\varepsilon)\|_X^q\big)^{1/q}
	\leq B^m \big(\mathbb{E}\|P(\varepsilon)\|_X^q \big)^{1/q}.
	\]
\end{lemma}
\begin{proof}
	For each $m$ we consider the functions $\{1, t,\ldots, t^m\}$ in $L_2(0,1)$. We show that there are polynomials $\{p_1^{(m)},\ldots,p_{m+1}^{(m)}\}$ of degree at most $m$ such that
	\[
	\int_{0}^1 t^{i-1} p_j^{(m)}(t) \, d t=\delta_{ij},
	\]
	for every $1\leq i,j\leq m+1$. Indeed, writing $p_j^{(m)}(t)=\sum_{k=1}^{m+1}a_{kj}^{(m)}t^{k-1}$ we get
	\[
	\delta_{ij}=\int_{0}^1 t^{i-1} p_j^{(m)}(t)  d t
	=\sum_{k=1}^{m+1}a_{kj}^{(m)}\int_{0}^1 t^{i+k-2}  d t
	=\sum_{k=1}^{m+1}\frac{1}{i+k-1}a_{kj}^{(m)},
	\]
	for every $1\leq i,j\leq m+1$. In other words,  we obtain the matrix identity
	\[
	I=HA,
	\]
	where $H$ is the well-known Hilbert matrix and $A$ is the matrix defined by the coefficients $a_{ij}^{(m)}$. Thus, we have $A=H^{-1}$, which provides a specific formula for the polynomials $p_j^{(m)}$.
It is easy to check that there is a constant $C >1$ so that  $\sup_{i,j}|a_{ij}^{(m)}| \leq C^{m}$ and therefore (taking $B=2C$)
	\[
	\sup_{0 < t <1} \vert p_{j}^{(m)} (t) \vert \leq (m+1) C^{m} \leq B^{m}\,.
	\]
	Notice that, if $Q$ is a  tetrahedral polynomial of degree $m$, then
	\[
	P_k(\varepsilon)=\int_{0}^1 P(t\varepsilon) p_{k+1}^{(m)}(t)  d t,
	\]
	for every $0\leq k\leq m$. So we get
	\[
	\big(\mathbb{E}\|P_k(\varepsilon)\|_X^q\big)^{1/q}
	\leq \int_{0}^1\big(\mathbb{E}\|P(t\varepsilon) p_{k+1}^{(m)}(t) \|_X^q\big)^{1/q} d t
	\leq B^m \int_{0}^1\big(\mathbb{E}\|P(t\varepsilon) \|_X^q\big)^{1/q} d t.
	%(\mathbb{E}\|Q_k(\varepsilon)\|_X^q)^{1/q}
	%&\leq C^m \int_{0}^1(\mathbb{E}\|Q(\varepsilon) \|_X^q)^{1/q} d t
	%= C^m (\mathbb{E}\|Q(\varepsilon) \|_X^q)^{1/q},
	\]
	Now,  \cite[Lemma~3.2.3]{DPGi99} gives
	\[
	\big(\mathbb{E}\|P(t\varepsilon) \|_X^q\big)^{1/q} \le \big(\mathbb{E}\|P(\varepsilon) \|_X^q\big)^{1/q}
	\]
	for every $0\le t\le 1$, and this completes the proof.
\end{proof}

For the last ingredient in the proof of Theorem~\ref{cotimplicapolcot}, we need a rather convoluted description of a polynomial in terms of the parity of the  exponents of the variables.
Fix an even $m\in\N$.
Given $A\subseteq  [n]$ we define
\[\Lambda_A=\{\alpha\in\N_0^n:\ |\alpha|=m, \alpha_i \text{ is odd if and only if } i\in A \}.\]
Since $m$ is even, it is clear that $\Lambda_A\neq \varnothing$ if and only if $A$ has even cardinality between 0 and $m$. In the rest of this discussion we only consider $A$ with $\Lambda_A\neq \varnothing$. Note that for any $\varepsilon\in\{-1,1\}^n$ and $z\in \T^n$, we have
\[(\varepsilon z)^\alpha= \varepsilon_A z^\alpha\] for every $\alpha\in \Lambda_A$, where, as always, $\varepsilon_A=\prod_{i\in A}\varepsilon_i$.

Now, for an  $m$-homogeneous polynomial of $n$ variables
$P(z)=\sum_{|\alpha|=m}x_\alpha z^\alpha$
we write
\[P_A(z)=\sum_{\alpha\in\Lambda_A}x_\alpha z^\alpha.\]
With this notation, we clearly have
\[P(\varepsilon z)=\sum_{A\subseteq  [n]}\varepsilon_A P_A(z)
%=\sum_{k=0}^{m/2}\sum_{A\subseteq \mathcal{A}_k}\varepsilon_A P_A(z),
.\]
As we can see from the expression above,  $P(\varepsilon z)$ regarded as a polynomial on $\varepsilon$ is tetrahedral.
Also, we may write $P(\varepsilon z)$ as the sum of its homogeneous components (as a function of $\varepsilon$). As we have already mentioned, each $A$ considered has even cardinality between 0 and $m$. So, if  we define
$$\mathcal{A}_k=\{A\subseteq  [n]: \ |A|=2k\},$$
we can write
\begin{equation}
\label{eq8}
P(\varepsilon z)
=\sum_{k=0}^{m/2}\sum_{A\in \mathcal{A}_k}\varepsilon_A P_A(z).
\end{equation}
Note that, whenever $i$ belongs to some $A$,  the exponents of $z_i$ are odd for every monomial in  $P_A(z)$.
Also, since $m$ is even, given $\alpha \in \Lambda_A$, we have that $\sum_{i\in A}\alpha_i$ must be even and greater than $|A|=2k$. We then define
$$\Lambda_{A,l}=\{\alpha\in\Lambda_A:\ \sum_{i\in A}\alpha_i=2l\},$$ which allows us to write, for $A\in \mathcal A_k$,
\begin{equation}
\label{eq9}
P_A( z)
=\sum_{l=k}^{m/2}\sum_{\alpha \in \Lambda_{A,l}} x_\alpha z^\alpha=\sum_{l=k}^{m/2}P_{A,l}(z).
\end{equation}
Note that $P_{A,l}(z)$ is the $2l$-homogeneous component of the polynomial $P_A(z)$  regarded as a function of the variables $z_i$ with $i\in A$ (that is, the variables with odd exponents).
In other words, the polynomial $P_{A,l}(z)$ consists of the monomials $x_\alpha z^\alpha$ of $P_A(z)$ where the sum of the odd exponents equals $2l$.

To conclude our description of $P$, for $\alpha\in \Lambda_{A}$ define exponents $\beta,\gamma$ and $1_A$ by
\begin{gather*}
\beta_i=\begin{cases}
0 &\text{if } i\in A
\\\frac{\alpha_i}{2} &\text{if } i\in A^c
\end{cases}, \quad
\gamma_i=\begin{cases}
\frac{\alpha_i-1}{2} &\text{if } i\in A
\\0 &\text{if } i\in A^c
\end{cases} \quad \text{and} \quad
 1_{A,i}=\begin{cases}
1 &\text{if } i\in A
\\0 &\text{if } i\in A^c
\end{cases},
\end{gather*}
for every $1\le i \le n$. Note that $\alpha=2\beta+2\gamma+1_A$ where $\beta\in \N_0^n$ is supported in $A^c$ and $\gamma,1_A\in \N_0^n$ are supported in $A$. Moreover, for $\alpha\in\Lambda_{A,l}$ we have
\begin{align*}
|\beta|&=\sum_{i=1}^n \beta_i=\sum_{i\in A^c} \frac{\alpha_i}{2}
=\frac{|\alpha|}{2}- \sum_{i\in A}\frac{\alpha_i}{2}= \frac{m}{2}- l,
\intertext{and}
|\gamma|&=\sum_{i=1}^n \gamma_i=\sum_{i\in A} \frac{\alpha_i-1}{2}
= \frac{2l-|A|}{2}=l-k.
\end{align*}
Denote the set of all the exponents $\beta$ supported in $A^c$ with $|\beta|=m/2-l$ by $B_{A,l}$ and the set of all the exponents $\gamma$ supported in $A$ with $|\gamma|=l-k$ by $\Gamma_{A,l}$. We get
\begin{equation}
\label{eq14}
P_{A,l}(z)= \sum_{\alpha \in \Lambda_{A,l}} x_\alpha z^\alpha
=\sum_{\gamma \in \Gamma_{A,l}}\sum_{\beta \in B_{A,l}} x_{2\beta+2\gamma+1_A} z^{2\beta+2\gamma+1_A}
=\sum_{\gamma \in \Gamma_{A,l}}\left(\sum_{\beta \in B_{A,l}} x_{2\beta+2\gamma+1_A} z^{2\beta}\right)z^{2\gamma+1_A}.
\end{equation}

Gathering \eqref{eq8}, \eqref{eq9} and \eqref{eq14} we get the full description of $P(\varepsilon z)$ proving the following lemma.

\begin{lemma} \label{lemma4}
 For an even $m\in\N$, an $m$-homogeneous polynomial in $n$ variables
 \[
P(z)=\sum_{|\alpha|=m}x_\alpha z^\alpha,
\]
 and $\varepsilon\in\{-1,1\}^n$ we have
 \[
P(\varepsilon z)
 =\sum_{k=0}^{m/2}\sum_{A\in \mathcal{A}_k}\sum_{l=k}^{m/2}\sum_{\gamma \in \Gamma_{A,l}}\sum_{\beta \in B_{A,l}} x_{2\beta+2\gamma+1_A} \varepsilon_A z^{2\beta+2\gamma+1_A}.
\]
\end{lemma}

With the same argument we may deduce a similar formula when $m$ is odd. For every $m$-homogeneous polynomial in $n$ variables $P$ and $\varepsilon\in\{-1,1\}^n$, we get
\[
P(\varepsilon z)
 =\sum_{k=0}^{(m-1)/2}\sum_{A\in \mathcal{A}'_k}\sum_{l=k}^{(m-1)/2}\sum_{\gamma \in \Gamma'_{A,l}}\sum_{\beta \in B'_{A,l}} x_{2\beta+2\gamma+1_A} \varepsilon_A z^{2\beta+2\gamma+1_A},
\]
where
\begin{align*}
\mathcal{A}'_k&=\{A\subseteq  [n]: \ |A|=2k+1\},
\\ \Gamma'_{A,l}&=\big\{\gamma \in \N_0^n:  \sum_{i\in A}\gamma_i=l-k\ \text{ and } \gamma_i=0 \text{ for }i \in A^c\big\}, \text{ and}
\\ B'_{A,l}&=\big\{\beta \in \N_0^n:  \sum_{i\in A}\beta_i=(m-1)/2-l \ \text{ and } \beta_i=0 \text{ for }i \in A\big\}.
\end{align*}

We are now in position to give the proof of Theorem~\ref{cotimplicapolcot}.

\begin{proof}[Proof of Theorem~\ref{cotimplicapolcot}]
First notice that \ref{proposition5-b}~$\Rightarrow$~\ref{proposition5-a} follows immediately by taking $m=1$ in \ref{proposition5-b}.
Next we show that~\ref{proposition5-a}~$\Rightarrow$~\ref{proposition5-c}.
Let $\{ x_{A} \colon A \subseteq [n], \, \vert A \vert \le m \}$ be a family of vectors in $X$. Applying Lemma~\ref{cotawalsh} to the subfamilies $\{ x_{A} \colon A \subseteq [n], \, \vert A \vert =k \}$ for each $0\leq k\leq m$ and denoting $C=4^{1/q}C_{q}(X)$ we get
\[
\sum_{A\subseteq [n]}\|x_A\|^q=\sum_{k=0}^m \sum_{\stackrel{A\subseteq [n]}{|A|=k}}\|x_A\|^q\leq \sum_{k=0}^m C^{qk} \int_{\T^n} \Big\|\sum_{\stackrel{A\subseteq [n]}{|A|=k}} x_A z_A\Big\|^q dz.\]
	Using Lemmas~\ref{lemma1} and~\ref{lemma3} we obtain
	\begin{multline*}
	\sum_{A\subseteq [n]}\|x_A\|^q
	\leq \sum_{k=0}^m ((1+\sqrt{2})C)^{qk} \E \Big\|\sum_{\stackrel{A\subseteq [n]}{|A|=k}} x_A \varepsilon_A\Big\|^q
	\leq B^m\sum_{k=0}^m ((1+\sqrt{2})C)^{qk} \E \Big\|\sum_{A\subseteq [n]} x_A \varepsilon_A\Big\|^q
	\\
	\leq (m+1) ((1+\sqrt{2})BC)^{qm} \E \Big\|\sum_{A\subseteq [n]} x_A \varepsilon_A\Big\|^q \leq (20BC_{q}(X))^{qm} \E \Big\|\sum_{A\subseteq [n]} x_A \varepsilon_A\Big\|^q.
	\end{multline*}
	This gives \eqref{shankar} and completes the argument.\\
It only remains to show that~\ref{proposition5-c}~$\Rightarrow$~\ref{proposition5-b}. As a first step we show that the inequality holds for homogeneous polynomials. Let $C_{\varepsilon}$ be the constant provided by \ref{proposition5-c} and $B$ the constant in Lemma~\ref{lemma3}. Our aim is to show that if $C=\max \{C_{\varepsilon}^{2}, B^{4}\}$, then
\begin{equation} \label{guiness}
\Big( \sum_{\vert \alpha \vert = m} \Vert x_{\alpha} \Vert^{q} \Big)^{1/q}
\leq C^{m} \bigg( \int_{\mathbb{T}^{n}} \Vert P(z) \Vert^{q} dz \bigg)^{1/q}
\end{equation}
for every $m$-homogeneous polynomial $P(z) = \sum_{\vert \alpha \vert = k}  x_{\alpha} z^{\alpha}$ of $n$ variables. We proceed by induction on $m$. The case $m=1$ is the well known equivalence between Rademacher and Steinhaus averages, which is a particular case of Lemma~\ref{lemma1}.
We fix some $m\ge 2$ and suppose that \eqref{guiness} holds for every $k$-homogeneous polynomial with $k < m$.
We may assume that $m$ is even (being the case when $m$ is odd completely analogous). Fix an $m$-homogeneous polynomial in $n$ variables
\[
P(z)=\sum_{|\alpha|=m}x_\alpha z^\alpha.
\]
Since our goal involves estimating an integral of $P(z)$, we take advantage of the rotation invariance and work with $P(\varepsilon z)$, but this requires some preparation.
 For a fixed $1\le k\le m/2$ and $A\subseteq [n]$ with $|A|=2k$, take $k\le l\le m/2$ and define $P_A$ and $P_{A,l}$ as in \eqref{eq9}.
 Intuitively, $P_{A,l}$ detaches the $z_i$'s with odd exponent from the $z_i$'s with even exponent. This enables us to use the inductive hypothesis twice (once for the odd and once for the even part) to assemble the polynomials $P_{A,l}$.
 Let $\T^{A^c}$ denote $|A^c|$ copies of the torus indexed in $A^c$. We get
 \begin{multline*}
 \sum_{\gamma \in \Gamma_{A,l}}\sum_{\beta \in B_{A,l}}
  \|x_{2\beta+2\gamma+1_A}\|^q
\le \sum_{\gamma \in \Gamma_{A,l}}C^{q(m/2-l)}\int_{\T^{A^c}} \Big\|\sum_{\beta \in B_{A,l}} x_{2\beta+2\gamma+1_A} z^{\beta}\Big\|^q dz
 \\
 \le C^{q(m/2-l)}\int_{\T^{A^c}} \sum_{\gamma \in \Gamma_{A,l}}\Big\|\sum_{\beta \in B_{A,l}} x_{2\beta+2\gamma+1_A} z^{\beta}\Big\|^q dz
 \\
 =C^{q(m/2-l)}\int_{\T^{A^c}} \sum_{\gamma \in \Gamma_{A,l}}\Big\|\sum_{\beta \in B_{A,l}} x_{2\beta+2\gamma+1_A} z^{2\beta}\Big\|^q dz,
 \end{multline*}
 where the last step follows by a change of variables. Since $\beta$ is supported in $A^c$, the variables $z_i$ with $i\in A$ do not appear in the expression above. So, we are still able to introduce them by applying the inductive hypothesis again. We obtain
 \begin{equation}
 \label{eq11}
 \begin{split}
 \sum_{\gamma \in \Gamma_{A,l}}\sum_{\beta \in B_{A,l}}
 \|& x_{2\beta+2\gamma+ 1_A}\|^q \\
 &\le C^{q(m/2-l)}C^{q(l-k)}\int_{\T^n}\Big\|\sum_{\gamma \in \Gamma_{A,l}}\Big(\sum_{\beta \in B_{A,l}} x_{2\beta+2\gamma+1_A} z^{2\beta}\Big)z^\gamma\Big\|^q dz
 \\ &= C^{q(m/2-k)}\int_{\T^n}\Big\|\sum_{\gamma \in \Gamma_{A,l}}\Big(\sum_{\beta \in B_{A,l}} x_{2\beta+2\gamma+1_A} z^{2\beta}\Big)z^{2\gamma}\Big\|^q dz
 \\ &= C^{q(m/2-k)}\int_{\T^n}\Big\|z^{1_A}\sum_{\gamma \in \Gamma_{A,l}}\Big(\sum_{\beta \in B_{A,l}} x_{2\beta+2\gamma+1_A} z^{2\beta}\Big)z^{2\gamma}\Big\|^q dz
 \\ &= C^{q(m/2-k)}\int_{\T^n}\|P_{A,l}(z)\|^q dz,
 \end{split}
 \end{equation}
 where in the last step we used \eqref{eq14}. Since $P_{A,l}$ is the $2l$-homogeneous component of $P_A$ regarded as a function depending only on the variables $z_i$ with $i\in A$, we have
 \begin{equation}
 \label{eq12}
  \int_{\T^n}\|P_{A,l}(z)\|^q dz\le \int_{\T^n}\|P_{A}(z)\|^q dz.
 \end{equation}
 From \eqref{eq11} and \eqref{eq12}, we deduce
 \begin{multline*}
  \sum_{l=k}^{m/2}\sum_{\gamma \in \Gamma_{A,l}}\sum_{\beta \in B_{A,l}}
  \|x_{2\beta+2\gamma+1_A}\|^q \le (\frac{m}{2}-k) C^{q(m/2-k)} \int_{\T^n}\|P_{A,l}(z)\|^q dz
  \\
  \le m  C^{q(m/2-k)}\int_{\T^n}\|P_{A}(z)\|^q dz.
 \end{multline*}
 Finally we deal with $P(\varepsilon z)$ using~\eqref{guiness} and \eqref{eq8}. Taking Lemma~\ref{lemma4} into consideration (and the definition of $C$) we get
 \begin{align*}
  \sum_{|\alpha|=m}
  \|x_{\alpha}\|^q =
  \sum_{k=1}^{m/2} & \sum_{A\in\mathcal{A}_k}\sum_{l=k}^{m/2}\sum_{\gamma \in \Gamma_{A,l}}\sum_{\beta \in B_{A,l}}
  \|x_{2\beta+2\gamma+1_A}\|^q \\
  & \le m \int_{\T^n} \sum_{k=1}^{m/2}C^{q(m/2-k)} \sum_{A\in\mathcal{A}_k} \|P_{A}(z)\|^q dz \\
  &  \le m \int_{\T^n} \sum_{k=1}^{m/2} C^{q(m/2-k)}C_\varepsilon^{2qk}
  \E\Big\|\sum_{A\in\mathcal{A}_k}\varepsilon_A P_{A}(z)\Big\|^q dz
  \\
  & \le m C^{qm/2} \int_{\T^n}  \sum_{k=1}^{m/2} \E\Big\|\sum_{A\in\mathcal{A}_k}\varepsilon_A P_{A}(z)\Big\|^q dz.
\end{align*}
Using now Lemma~\ref{lemma3}, we have (note that $m^{2} \leq B^{qm}$ for every $m$, since $B>2$ )
\begin{multline*}
\sum_{|\alpha|=m}
  \|x_{\alpha}\|^q
 \le m C^{qm/2} B^{qm} \int_{\T^n} \sum_{k=1}^{m/2}  \E\|P(\varepsilon z)\|^q dz \\
\le m^2 C^{qm/2} B^{qm} \E\int_{\T^n} \|P(\varepsilon z)\|^q dz
\le C^{qm} \int_{\T^n} \|P(z)\|^q dz \,.
\end{multline*}
Hence \eqref{guiness} (and then \eqref{eq3}) holds for every $m$-homogeneous polynomial. To finish the argument, take an arbitrary polynomial of degree $m$
\[
P(z)=\sum_{|\alpha|\le m}x_\alpha z^\alpha.
\]
For $0\le k\le m$, denote by $P_k$ its $k$-homogeneous projection. We have
 \begin{align*}
 \sum_{|\alpha|\le m}\|x_\alpha\|^q
 &= \sum_{k=0}^m \sum_{|\alpha|=k}\|x_\alpha\|^q
 \leq \sum_{k=0}^m C^{qk} \int_{\T^n} \|P_k(z)\|^q dz
 \leq \sum_{k=0}^m C^{qk} \int_{\T^n} \|P(z)\|^q dz
 \\ &\leq (m+1) C^{qm} \int_{\T^n} \|P(z)\|^q dz
 \leq (2C)^{qm} \int_{\T^n} \|P(z)\|^q dz .
 \end{align*}
This concludes the argument.
\end{proof}

\bibliographystyle{abbrv}
\bibliography{biblio_CaMaSP}

\end{document}